\documentclass[12pt,a4paper]{amsart}

\usepackage{amsfonts,amssymb,enumerate,bm,hhline,makecell,array,mathtools}
\usepackage{mathrsfs}
\usepackage{comment}
\usepackage[normalem]{ulem}
\usepackage[colorlinks=true,citecolor=blue!70!black, urlcolor=blue!70!black, pagebackref, linkcolor=blue!70!black]{hyperref}
\usepackage{tikz-cd}
\usepackage{tikz}
\usetikzlibrary{arrows.meta}
\usepackage{amscd}
\usepackage{stmaryrd}
\usepackage[shortlabels]{enumitem}
\setlist[itemize]{noitemsep,nolistsep}
\usepackage{tensor}
\usepackage{tikz}
\usetikzlibrary{matrix,arrows}
\usepackage{extarrows}
\usepackage[toc,page]{appendix}
\usepackage[all,ps,cmtip,rotate]{xy} 
\usepackage{color}
\usepackage{faktor}
\usepackage{hyperref}
\usepackage[capitalize, noabbrev]{cleveref}
\hypersetup{
    colorlinks,
    citecolor = blue!50!black,
    linkcolor=blue!50!black,
    urlcolor=black
}

\makeatletter
\def\@tocline#1#2#3#4#5#6#7{\relax
  \ifnum #1>\c@tocdepth 
  \else
    \par \addpenalty\@secpenalty\addvspace{#2}%
    \begingroup \hyphenpenalty\@M
    \@ifempty{#4}{%
      \@tempdima\csname r@tocindent\number#1\endcsname\relax
    }{%
      \@tempdima#4\relax
    }%
    \parindent\z@ \leftskip#3\relax \advance\leftskip\@tempdima\relax
    \rightskip\@pnumwidth plus4em \parfillskip-\@pnumwidth
    #5\leavevmode\hskip-\@tempdima
      \ifcase #1
       \or\or \hskip 1em \or \hskip 2em \else \hskip 3em \fi%
      #6\nobreak\relax
    \dotfill\hbox to\@pnumwidth{\@tocpagenum{#7}}\par
    \nobreak
    \endgroup
  \fi}
\makeatother

\usepackage{enumitem}
\setlist[itemize]{noitemsep,nolistsep}
\setlist[enumerate]{noitemsep,nolistsep}
\usepackage{colortbl} 
\usepackage{leftidx}

\allowdisplaybreaks

\def\Z{{\bf Z}}
\def\Zsheaf{{\underline{\bf Z}}}

\def\C{{\bf C}}

\def\Q{{\bf Q}}
\def\P{{\bf P}}

\def\hk{hyper-K\"ahler}

\def\hkm{hyper-K\"ahler manifold}

\def\cE{\mathcal{E}}
\def\cF{\mathcal{F}}

\def\cM{\mathcal{M}}

\def\cO{\mathcal{O}}

\def\-{\textup{-}}

\def\lra{\longrightarrow}
\def\llra{\hbox to 10mm{\rightarrowfill}}
\def\lllra{\hbox to 15mm{\rightarrowfill}}

\def\llla{\hbox to 10mm{\leftarrowfill}}
\def\lllla{\hbox to 15mm{\leftarrowfill}}

\newcommand\vsim{\rotatebox[origin=cc] {-90}{$\sim$}}
\newcommand\simHalf{\rotatebox[origin=cc] {-45}{$\sim$}}

\DeclareMathOperator{\Aut}{Aut}

\DeclareMathOperator{\Br}{Br}

\DeclareMathOperator{\Coker}{Coker}
\DeclareMathOperator{\coker}{Coker}

\DeclareMathOperator{\Ext}{Ext}
\DeclareMathOperator{\Exc}{Exc}

\DeclareMathOperator{\Hilb}{Hilb}

\DeclareMathOperator{\Id}{Id}
\DeclareMathOperator{\id}{Id}
\def\Im{\mathop{\rm Im}\nolimits}
\def\im{\mathop{\rm Im}\nolimits}

\DeclareMathOperator{\Ker}{Ker}

\DeclareMathOperator{\Kum}{Kum}

\DeclareMathOperator{\Mod}{Mod}

\DeclareMathOperator{\NS}{NS}

\DeclareMathOperator{\Pic}{Pic}

\DeclareMathOperator{\pr}{\mathsf{pr}}
\DeclareMathOperator{\Proj}{Proj}

\DeclareMathOperator{\Sh}{Sh}

\DeclareMathOperator{\tors}{tors}

\def\llra{\hbox to 10mm{\rightarrowfill}}
\def\lllra{\hbox to 15mm{\rightarrowfill}}

\def\bw#1#2{\textstyle{\bigwedge\hskip-0.9mm^{#1}}\hskip0.2mm{#2}}

\def\subset{\subseteq}

\newtheorem{lem}{Lemma}[section]

\newtheorem{thm}[lem]{Theorem}
\newtheorem{cor}[lem]{Corollary}
\newtheorem{prop}[lem]{Proposition}

\newtheorem*{thm*}{Theorem}

\theoremstyle{definition}
\newtheorem{defin}[lem]{Definition}
\newtheorem{rem}[lem]{Remark}

\theoremstyle{remark}
\newtheorem*{remark*}{Remark}
\newtheorem*{note*}{Note}

\def\Lkkk[#1]{{\Lambda_{\KKK^{[#1]}}}}
\def\kkk[#1]{{\KKK^{[#1]}}}
\DeclareMathOperator{\KKK}{{K3}}

\def\sss[#1]{{S^{[#1]}}}

\definecolor{orange}{rgb}{1,0.55,0}
\definecolor{purple}{RGB}{148,55,255}
\definecolor{myblue}{RGB}{4,51,255}

\title{On Brauer groups of known Enriques manifolds}
\author[A.\ Frassineti]{Alessandro Frassineti}
\address{Università di Genova, Via Dodecaneso 35, 16146 Genova, Italy}
\email{alessandro.frassineti@edu.unige.it}

\author[F.\ Rizzo]{Francesca Rizzo}
\address{Université Paris Cité and Sorbonne Université, CNRS, IMJ-PRG, F-75013 Paris, France}
\email{francesca.rizzo@imj-prg.fr}

\author[F.\ Tufo]{Federico Tufo}
\address{Università di Bologna, Piazza di Porta San Donato 5, Bologna, Italy}
\email{federico.tufo2@unibo.it}

\author[M.\ Verni]{Matteo Verni}
\address{Sorbonne Université, IMJ-PRG, 4 place Jussieu, 75005 Paris, France}
\email{verni@imj-prg.fr}
\thanks{This project has received funding from the European
Research Council (ERC) under the European
Union's Horizon 2020 research and innovation
programme (ERC-2020-SyG-854361-HyperK) and from INdAM - GNSAGA Project, Varietà di Fano e hyperkahler: costruzioni, classificazioni e collegamenti (CUP E53C24001950001). }

\begin{document}
\begin{abstract}
We compute the Brauer group of some of the known Enriques manifolds. We then build special Brauer--Severi varieties on these manifolds and study the pull-back map from the Brauer group of an Enriques manifold to that of its hyper-K\"ahler universal cover, from both a geometric and an algebraic perspective.
\end{abstract}
\maketitle
\tableofcontents

\section{Introduction}

An Enriques surface is a smooth projective surface obtained as the quotient of a K3 surface by an involution without fixed points. These surfaces occupy a prominent position in the Enriques-Kodaira classification of algebraic surfaces as they have zero Kodaira dimension and trivial first Betti number but non-trivial canonical bundle.

The geometry of Enriques surfaces is deeply intertwined with that of K3 surfaces.
For this reason, the flourishing study of \hk \ geometry opened the way to extend the rich theory of Enriques surfaces to higher dimension, at least over the complex numbers.

In what follows we adopt the definition of \cite{OguisoSchroer2011}, which has been extended to the mildly singular setting in \cite{DRTX26}.
Another possible definition has been given by Boissière, Nieper--Wisskirchen, Sarti in \cite{BNWS10-DefEnriques}.

\begin{defin}[\cite{OguisoSchroer2011}]\label[defin]{OSdef}
    An \textit{Enriques manifold} is a connected complex manifold \(T\) with non-trivial fundamental group whose universal cover is a compact \hk \ manifold \(X\). 
    We call \textit{index} of \(T\) the degree of the universal cover \(\pi \colon X\to T\).
\end{defin}
One can show (see for example \cite[Section 2]{OguisoSchroer2011}) that $\pi_1(T)$ is a cyclic group whose action on $X$ is purely non-symplectic. In particular, all Enriques manifolds are projective. Moreover, the index \(d=|\pi_1(T)|\) must divide $\dim(X)+1$, by multiplicativity of the holomorphic Euler characteristic.

\subsection{Known constructions}\label{subsec_known_constructions}

The following are all the known constructions of Enriques manifolds: they are obtained from \emph{natural} automorphisms on Hilbert schemes of points on a K3 surface, or on generalized Kummer varieties.

\begin{enumerate}[label =\rm (\alph*)]
    \item\label{item_E_n} (\cite[Proposition 4.1]{OguisoSchroer2011}) Let \(n\) be an odd positive integer and \(i\) be a fixed point free involution on a K3 surface \(S\). It induces an involution \(i^{[n]}\) on \(\text{Hilb}^n(S)\) which is itself fixed point free since \(n\) is odd. The quotient manifold \[E_n\coloneqq \faktor{\text{Hilb}^n(S)}{\langle i^{[n]}\rangle }\] is an Enriques manifold of dimension \(2n\) and index \(d=2\).
    By \cite[Proposition 3.1]{Oguiso-Schroer_Periods}, any small deformation of the Enriques manifold $E_n$ is again of this form.
    \item\label{item_K_n} (\cite[Proposition 4.2]{OguisoSchroer2011}) Let \(n\geq3\) be an odd integer, and let \(C\) and \(C'\) be two elliptic curves. Set \(A\coloneqq C\times C'\), choose a point \(a'\in C'\) of order \(2\) and any point \(a\in C\). The map 
    \[
        \iota:A\to A, \ (b,b')\mapsto(-b+a,b'+a')
    \]
    induces an involution \(\iota^{[n+1]}\) on \(A^{[n+1]}\). When $a\in C$ is a torsion point of order $n+1$, the automorphism $\iota^{[n+1]}$ restricts to a fixed point free involution $\iota^{{\llbracket n\rrbracket}}$ of $\Kum_n(A)$.
  
    In this case, the quotient
    \[K_n\coloneqq \faktor{\text{Kum}_n(A)}{\langle \iota^{\llbracket n\rrbracket}\rangle }\] is an Enriques manifold of dimension \(2n\) and index \(2\). 
\item\label{item_T_n} Consider the abelian surface $A = C\times C'$ where \(C'\) is an elliptic curve, and \(C\) is an elliptic curve admitting an automorphism $h_d\in \Aut_\Z(C)$ of order \(d=3\).
    Given points \(a_1\), \(a_2\in C\) of order $d$, consider the automorphism 
    \begin{align*}
        \psi_d \coloneqq t_a\circ ( h_d\times \id_{C'})\colon A&\lra A,\\
        (b_1, b_2) &\longmapsto (h_d(b_1)+a_1, b_2+a_2)
    \end{align*}
    where $t_a$ is the translation by the point $a\coloneqq(a_1,a_2)$.
    By \cite[Proposition 6.4]{OguisoSchroer2011}, there exists an appropriate choice of \(a_1\) such that the group \(G\coloneqq \langle\psi_d^{\llbracket 3m-1\rrbracket}\rangle\) acts freely on \(\text{Kum}_{3m-1}(A)\) and \[T_{3m-1}\coloneqq \faktor{\text{Kum}_{3m-1}(A)}{G}\] is an Enriques manifold of dimension \(6m-2\) and index $3$.

A similar construction starting from an elliptic curve \(C\) admitting an automorphism $h_d\in \Aut_\Z(C)$ of order \(d=4\) yields Enriques manifolds
    \[R_{4m-1}\coloneqq \faktor{\text{Kum}_{4m-1}(A)}G\]of dimension \(8m-2\) and index $4$.
\end{enumerate}
In \cite[Proposition 4.2]{PacSar23}, Pacienza and Sarti  provided the list of all possible indices for Enriques manifolds obtained as quotient of \hk \ manifolds of \(\text{K3}^{[n]}\)-type and \(\text{Kum}_n\)-type. 
In an upcoming article, Macrì and Mongardi prove that all the Enriques manifolds whose universal cover is a \hk \ manifold of \(\text{K3}^{[n]}\)-type have index $2$.

Moreover, in \cite{BGGG25} Billi, F. Giovenzana, L. Giovenzana, and Grossi proved that there are no Enriques manifolds arising from \hk \ manifolds of OG10-type. This leaves OG6-type manifolds as the last known deformation family of \hk\ manifolds for which nothing about related Enriques manifolds has yet been established.

\subsection{Main results}
The goal of our work is to study the Brauer group of known Enriques manifolds. As is explained at the beginning of Section 3, for any Enriques manifold \(T\) one has
\[\text{Br}(T)=H^3(T,\Z)_{\tors}.\]
In particular, the Brauer group of an Enriques manifold is purely topological and thus invariant under deformation. 
For any Enriques surface $E$, Poincaré duality gives $\Br(E)\simeq \Z/2\Z$.

Our first main result is the following. 
\begin{thm*}[\cref{thm_brauer_computations}]
Let $n$ and $m$ be odd positive integers. In the above notations,  
\[\Br(E_n)=\Z/2\Z ,\]
    \[\Br(T_{3m-1})= \Z/3\Z \oplus \Z/3\Z.\]
\end{thm*}

For any Enriques manifold \(T\), the Brauer group can be expressed, using the Cartan--Leray spectral sequence on the universal cover \(X\to T\), as an extension of a subgroup of \((H^3(X,\Z)_{\tors})^{\pi_1(T)}\) by the kernel of a morphism 
$$
d_3\colon H^1(\pi_1(T),H^2(X,\Z)) \lra \pi_1(T),
$$
arising as a differential at the third page of the Cartan spectral sequence, see \cref{prop_sequences_for_AMinv}.

The cohomology group \(H^1(\pi_1(T),H^2(X,\Z))\) is computed for all known Enriques manifolds in \cref{sec:examples}. However, determining precisely the above differential is more involved. We compute $d_3$ for all Enriques manifolds obtained as quotients  of $\Hilb^n(S)$ or $\Kum_n(A)$ by natural automorphisms in \cref{lem_E^40_S} and \cref{lem_E^40_Kum} respectively, hence in particular for all known Enriques manifolds.
 
On the other hand, $H^3(X,\Z)_{\tors}$ remains poorly understood even for known \hkm s. It is expected to always vanish, however this has been confirmed only for the deformation types $K3^{[n]}$ \cite{Markman2007_integral} and $\Kum_{2n}$ (\cite{KapferMenet2018} for $n=1$, \cite{Hartlieb_Verni25} in general). The known Enriques manifolds whose universal cover belongs to one of these deformation types are precisely $E_n$ and $T_{3m-1}$ for odd $m$, which is why these are the manifolds appearing in \cref{thm_brauer_computations}.

In the second part of the paper, we study the map
\[
    \pi^*_{\Br}:\text{Br}(T)\to\text{Br}(X),
\]
where \(T\) is an Enriques manifold and \(X\) is the \hk\ manifold associated to \(T\). This was done for Enriques surfaces by Beauville in \cite{Bea09-Brauer_of_Eriques_surf}.

In the case where \(T\) is of type \(E_n\), it is possible to give a characterization of those \(T\) for which the image of the generator of \(\text{Br}(E_n)\) through \(\pi^*_{\Br}\) is zero.

\begin{thm*}[\cref{corEquivalenceBrauerClass}]
    Let $\pi\colon S^{[n]}\to E_n$ be the Enriques manifold obtained from a K3 surface $S$ with a fixed point free involution $i$, and denote by $\sigma:=i^{[n]}$ the induced involution on $S^{[n]}$.
    Let $b_n\in \mathrm{Br}(E_n)$ be the non-zero class. The following are equivalent:
    \begin{enumerate}[label = \rm (\arabic*)]
        \item $\pi^*_{\Br}b_n=0$ in $\mathrm{Br}(S^{[n]})$;
        \item there exists $L\in \mathrm{Pic}(S^{[n]})$ satisfying $\pi_*c_1(L)=0$ and $c_1(L)\not\in (1-\sigma^*)\left(H^2(S^{[n]},\Z)\right)$;
        \item there exists $L\in \mathrm{Pic}(S^{[n]})$ satisfying $\sigma^*L=L^\vee$ and $q_{S^{[n]}}(L)\equiv 2 \pmod{4}$.
    \end{enumerate}
    Here, $q_{S^{[n]}}$ denotes the Beauville--Bogomolov--Fujiki form on $H^2(S^{[n]},\Z)$.
\end{thm*}

Moreover, in \cref{corEquivalenceBrauerClassModuli}, we show that these equivalent conditions define a locus of codimension one in the moduli space of Enriques manifolds of type $E_n$. These results extend \cite[Corollary 5.7 and Corollary 6.5]{Bea09-Brauer_of_Eriques_surf} to the case \(n \geq 2\).

In \cref{secSBvar}, we consider the morphism $\pi^*_{\Br}$ for any Enriques manifold. 
Our main result is \cref{charactSBV}, where we give a geometric characterisation of the classes in the kernel of $\pi^*_{\Br}$ by explicitly constructing the corresponding non-trivial Brauer--Severi varieties over $T$. This provides a concrete geometric counterpart to the algebraic criterion of \cite[Proposition~4.1]{Bea09-Brauer_of_Eriques_surf} and extends to all dimensions the construction carried out for surfaces in \cite[Lemma~10]{Martinez}.

\subsection*{Acknowledgements}
This project started during the summer school PRAGMATIC 2025 in Catania;
we thank the organisers Elena Guardo, Alfio Ragusa, Francesco Russo, Giovanni Staglianò and Giuseppe Zappalà for their kind hospitality, the stimulating research environment and the financial support during the school. We warmly thank Paolo Stellari for suggesting this problem to us, and Luigi Martinelli and James Hotchkiss for useful discussion at the school. We thank Burt Totaro for suggesting the strategy to prove \cref{torsH3}.

\section{Spectral sequences and group cohomology}\label[section]{sec:ss}

\subsection{The Cartan--Leray spectral sequence}\label{subsec_Carta_Leray} A standard introductory reference for spectral sequences is \cite[Chapter 9]{Voisin2002I}. Here we recall the construction of the Cartan--Leray spectral sequence.

Let \(M\) be a locally contractible topological space and let \(G\) be a finite group acting on \(M\). Taking the Grothendieck spectral sequence associated to the composition of functors
\begin{equation}\label{eq_sections_then_invariants}
  G\textrm{-}\Sh(M)\xrightarrow{\Gamma_G(M,-)} G\textrm{-}\Mod\xrightarrow{(-)^G} \Mod,
\end{equation}
and applied to the sheaf \(\Zsheaf_{M}\in G\textrm{-}\Sh(M)\),
we obtain a spectral sequence $E(M, G)$ whose second page is 
\begin{equation}\label{eq_E2page_2}
E^{p,q}_2(M, G)=H^p(G,H^q(M,\Z)).
\end{equation}

If \(f \colon M \to M'\) is a \(G\)-equivariant continuous map of locally contractible spaces with \(G\)-action, by the functoriality of the Grothendieck spectral sequence (\cite[Proposition 4.8]{Voisin2002II}) the natural arrow \(\underline{\Z}_{M'} \to Rf_*\underline{\Z}_M\) in \(D^+(G\textrm{-}\Sh(M'))\) induces a morphism of spectral sequences
\[E(M',G)\to E(M,G).\]
Since, under the identifications \(H_{\textup{sing}}^q(M,\Z)\simeq H^q(M,\underline{\Z}_M)\), the arrow \(\underline{\Z}_{M'} \to Rf_*\underline{\Z}_M\) induces precisely the Betti cohomology pullback, one can check that the second page arrows \[H^p(G,H^q(M',\Z)) \to H^p(G,H^q(M,\Z))\] are the ones induced in group cohomology by Betti cohomology pullback maps
\[f^* \colon H^q(M',\Z) \to H^q(M,\Z)\]

If moreover the group action is \textit{free} and we set \(N\coloneqq  M/G\), the pullback along the quotient map \(\pi \colon M\to N\) induces an equivalence of categories
\(\Sh(N) \xrightarrow{\pi^*} G\textrm{-}\Sh(M)\). Up to this equivalence, the composition of the two functors in \eqref{eq_sections_then_invariants} is precisely the global section functor
\[\Sh(N)\xrightarrow{\Gamma(N,-)} \Mod,\]
which means the abutment of the spectral sequence is isomorphic to
\begin{equation}\label{eq:abutment_2}
E^n(M, G)=H^{n}(N,\Z)\end{equation}
and the maps between abutments induced by \(f\colon M\to M'\) are the Betti cohomology pullbacks  
\[f^* \colon H^n(M'/G,\Z)\to H^n(M/G,\Z).\]
\subsection{Group cohomology}
We recall a computation in group cohomology which we will use systematically. For a group $G$ and a $G$-module $A$, we can consider the sub-module 
\[
A^G \coloneqq  \{a \in A \mid g\cdot a = a \text{ for any } g \in G\}
\]
of $G$-invariant elements of $A$.
This gives a functor \((-)^G\) which is left-exact but not right-exact. In particular, we can define the group cohomology of $G$ as the right derived functor of taking $G$-invariants:
\[
H^i(G,A) \coloneqq  R^i(-)^G (A).
\]
In this work, we mainly consider finite cyclic groups. In this case, it is possible to compute group cohomology as described in the following lemma.

\begin{lem}\label[lem]{lem_cyclic_group_cohom}
Let \(G\) be a cyclic group of order \(d\). Let \(\sigma\) be a generator of $G$ and let \(A\) be any \(G\)-module. Consider moreover the endomorphism \[N\coloneqq \sum_{i=0}^{d-1}\sigma^i\]
of $A$. For any \(p\geq 1\), the cohomology groups $H^p(G,A)$ are equal to
\[H^p(G,A)=\begin{cases}
    \faktor{\Ker N}{\Im (1-\sigma)} & \textup{if }p \textup{ is odd,}\rule[-13pt]{0pt}{35pt}\\
    \faktor{\Ker (1-\sigma)}{\Im N} & \textup{if }p \textup{ is  even.} 
    
\end{cases}\]
In particular, when \(A=\Z\) with the trivial \(G\)-action, for any \(p\geq 1\), we get
\[H^p(G,\Z)=\begin{cases}
    0 & \textup{if }p \textup{ is odd,}\\
    G & \textup{if }p \textup{ is even.}
\end{cases}\]
\end{lem}

\section{Brauer group of Enriques manifolds}\label[section]{secBrEnrM}
In this section, we recall the notion of Brauer group, which, under our assumptions, corresponds to the \textit{cohomological} Brauer group. Using the geometry of Enriques manifolds, or more precisely of their universal covers, we study the behaviour of the differentials in the Cartan--Leray spectral sequence associated to the action of the group of deck transformations. 
\begin{defin}
    For a smooth complex projective variety $X$ we define the \textit{Brauer group} as 
    $$\mathrm{Br}(X) \coloneqq  H^2(X_{\rm \acute et},\mathcal{O}_X^*).$$
\end{defin}
\begin{rem}
    Under our assumptions, this definition of the Brauer group coincides with the torsion subgroup of $H^2\left(X_{\rm an},\mathcal{O}_{X_{an}}^*\right)$, see \cite{Schroer_Brauer_group}. From now on, we always consider complex varieties with the analytic topology.
\end{rem}
In order to compute Brauer groups, we make use of the following exact sequences. For any positive integer \(n\) and any locally contractible topological space \(X\), the short exact sequence of abelian groups 
\[
0\longrightarrow \underline{\Z}_X \xrightarrow{\cdot n} \underline{\Z}_X \longrightarrow \underline{\Z/n\Z}_X\longrightarrow 0,
\]
yields
\begin{equation}\label{degree1_second_exact_seq_Br}
    0\longrightarrow H^1(X,\Z)\otimes\Z/n\Z \longrightarrow H^1(X,\Z/n\Z) \longrightarrow H^2(X,\Z)[n] \longrightarrow 0.
\end{equation}
and
\begin{equation}\label{second_exact_seq_Br}
    0\longrightarrow H^2(X,\Z)\otimes\Z/n\Z \longrightarrow H^2(X,\Z/n\Z) \longrightarrow H^3(X,\Z)[n] \longrightarrow 0.
\end{equation}
For \(X\) a complex manifold, we also have the Kummer exact sequence
\[
0\longrightarrow \underline{\Z/n\Z}_X \longrightarrow \mathcal{O}^*_X \longrightarrow \mathcal{O}^*_X \longrightarrow 0,
\]
from which one obtains
\begin{equation}\label{first_exact_seq_Br}
    0\longrightarrow \mathrm{Pic}(X)\otimes\Z/n\Z \longrightarrow H^2(X,\Z/n\Z) \longrightarrow \mathrm{Br}(X)[n] \longrightarrow 0.
\end{equation}

    The first Chern class $c_1\colon \Pic(X) \to H^2(X, \Z)$ induces a morphism between the short exact sequences \eqref{first_exact_seq_Br} and \eqref{second_exact_seq_Br}, from which by the Snake Lemma we obtain the exact sequence 
\[0\longrightarrow \frac{H^2(X,\Z)}{c_1(\Pic(X))}\,\otimes\Z/n\Z \longrightarrow \mathrm{Br}(X)[n] \longrightarrow H^3(X,\Z)[n] \longrightarrow 0.\]
Taking the colimit over all \(n\geq 1\), this becomes
\begin{equation}\label{eq_third_exact_se_Br}
0\longrightarrow \frac{H^2(X,\Z)}{c_1(\Pic(X))}\,\otimes\Q/\Z \longrightarrow \mathrm{Br}(X) \longrightarrow H^3(X,\Z)_{\tors}\longrightarrow 0.
\end{equation}

We call the term \(H^3(X,\Z)_{\tors}\) the \textit{topological Brauer group}. When $c_1$ is surjective, the Brauer group of $X$ is therefore isomorphic to $H^3(X, \Z)_{\tors}$.

Consider an Enriques manifold $\pi\colon X\to T$ and denote by $\psi$ a generator of $\pi_1(T)$. Since the action of $\psi$ on $X$ is purely non-symplectic, one has \(h^{2,0}(T)=0\) and thus the first Chern class \(c_1\colon \Pic(T)\to H^2(T,\Z)\) is surjective. Hence, the Brauer group of $T$ is isomorphic to $H^3(T, \Z)_{\tors}$. We study it using the Cartan--Leray spectral sequence \(E(X, G)\) for the morphism $\pi$ defined in \cref{sec:ss}, where $G\coloneqq  \langle \psi \rangle$.

Since $X$ is a \hk\ manifold, we have \(E^{p,1}_2(X, G)=H^{p}(G, H^1(X, \Z))=0\) for all $p$. By \cref{lem_cyclic_group_cohom}, the page $E_2(X,G)$ takes the form
\begin{equation*}
    \begin{tikzcd}
        \textcolor{gray}{q = 3}&  H^3(X, \Z)^{G} \ar[drr, "d_2^{0,3}"] & * & * & * &  * & \cdots\\
        \textcolor{gray}{q = 2}&  H^2(X, \Z)^{G} & V_1 & V_2 & V_1 & V_2 &\cdots\\
        \textcolor{gray}{q = 1}&  0 & 0 & 0 & 0 & 0 & \cdots\\
        \textcolor{gray}{q = 0}&  H^0(X, \Z)^{G} & 0 & \Z/d\Z & 0 & \Z/d\Z  &\cdots\\
        &\textcolor{gray}{p=0} &\textcolor{gray}{p=1} & \textcolor{gray}{p=2} & \textcolor{gray}{p=3} & \textcolor{gray}{p=4},
    \end{tikzcd}
\end{equation*}
where $V_i = H^i(G, H^2(X, \Z))$ for $i=1,2$.
Therefore, the third page looks like:
\begin{equation*}
    \begin{tikzcd}
        \textcolor{gray}{q = 3}&  \Ker d_2^{\,0,3} \ & * & * & * & * &\cdots\\
        \textcolor{gray}{q = 2}&  H^2(X, \Z)^{G} & V_1 \ar[rrrdd, "d_3^{1,2}"] & \faktor{V_2}{\im d_2^{0,3}} & * & * &\cdots\\
        \textcolor{gray}{q = 1}&  0 & 0 & 0 & 0 & 0&\cdots\\
        \textcolor{gray}{q = 0}&  H^0(X, \Z)^{G} & 0 & \Z/d\Z & 0 & \Z/d\Z &\cdots\\
        &\textcolor{gray}{p=0} &\textcolor{gray}{p=1} & \textcolor{gray}{p=2} & \textcolor{gray}{p=3} & \textcolor{gray}{p=4},
    \end{tikzcd}
\end{equation*}

This analysis yields the following description of the third degree cohomology of \(T\).
\begin{prop}\label[prop]{prop_sequences_for_AMinv}
    The Cartan-Leray spectral sequence associated to \(\pi\colon X\rightarrow T\) induces a short exact sequence
    \begin{equation}\label{eq_H3_SES}
\begin{tikzcd}
	0 & {\Ker d^{1,2}_3} & {H^3(T,\Z)} & {\Ker d^{0,3}_2} & 0.
	\arrow[from=1-1, to=1-2]
	\arrow[from=1-2, to=1-3]
	\arrow[from=1-3, to=1-4]
	\arrow[from=1-4, to=1-5]
\end{tikzcd}
    \end{equation}
    If \(H^3(X,\Z)^G\) is torsion free, the above sequence splits and we get the exact sequence
\begin{equation}\label{eq_Artin_Mumford_SES}
\begin{tikzcd}[sep = small]
	0 & {H^3(T,\Z)_{\tors}} & {H^1(G,H^2(X,\Z))} & {H^4(G,\Z)} & {E^{4,0}_4(X, G)} & 0.
	\arrow[from=1-1, to=1-2]
	\arrow[from=1-2, to=1-3]
	\arrow["{d^{1,2}_3}", from=1-3, to=1-4]
	\arrow[from=1-4, to=1-5]
	\arrow[from=1-5, to=1-6]
\end{tikzcd}\end{equation}
\end{prop}

\begin{proof}
    The exactness of \eqref{eq_H3_SES} follows from examining the third page of the above spectral sequence. If moreover \(H^3(X,\Z)^G\) is torsion free then so is \(\Ker d_2^{3,0}\), hence \eqref{eq_H3_SES} splits. By \cite[Theorem 6.5.8]{Weibel-intro_to_hom_algebra}, \(V_1\) is a torsion group hence its subquotient \(\Ker d_3^{1,2}\) is also torsion. Taking the torsion part of \eqref{eq_H3_SES} then yields \eqref{eq_Artin_Mumford_SES}.
\end{proof}

Therefore, in order to compute the Brauer group of $T$, we need to understand the differential \(d^{1,2}_3\). 
The rest of the section is devoted to studying this differential for Enriques manifolds obtained as quotients of $S^{[n]}$ or $\Kum_n(A)$ 
under a natural automorphism. This is done in \cref{lem_E^40_S} and \cref{lem_E^40_Kum} respectively.

We start by introducing some preliminary results.
\begin{thm}[\cite{Totaro_2020_integral}]\label[thm]{lem_torsion_free_Hilbert}
    Let \(S\) be a smooth projective surface with torsion free cohomology. Then the Hilbert schemes \(S^{[n]}\) have torsion free cohomology for all \(n\geq 1\).
\end{thm}
The case where \(K_S\simeq \cO_S\), which is what we need, follows already from an earlier result of Markman (\cite{Markman2007_integral}): we quote Totaro's theorem too to showcase the state of the art. 

\begin{lem}\label[lem]{lem_basta_svp}
Let \(S\) be a compact complex surface such that \(H^*(S,\Z)\) has no torsion and let \(\Delta \colon S \to S^n\) be the diagonal inclusion. There exists an exact sequence
\begin{equation}\label{eq_SES_1}
   0 \longrightarrow \bw2 H^1(S,\Z) \longrightarrow H^2(S^{n},\Z)^{\mathfrak{S}_n} \xrightarrow{\Delta^*}H^2(S,\Z) \longrightarrow H^2(S,\Z)\otimes \Z/n\Z \longrightarrow 0.
\end{equation}
\end{lem}
\begin{proof}
The lemma is equivalent to proving
that the morphism $\Delta^*\colon H^2(S^n, \Z)^{\mathfrak{S}_n}\to H^2(S, \Z)$ has kernel isomorphic to $\bigwedge^2H^1(S, \Z)$ and image $nH^2(S, \Z)$.

The K\"unneth formula gives a decomposition 
$$H^2(S^n, \Z) = \bigoplus_{i=1}^n \pr_i^*H^2(S,\Z) \ \oplus \bigoplus_{1\le i < j \le n} \pr_i^*H^1(S, \Z) \otimes \pr_j^*H^1(S, \Z),$$
where $\pr_i \colon S^n \to S$ is the projection to the $i$-th factor, and the action of $\mathfrak{S}_n$ is given by permuting factors. 
Hence we can write the $\mathfrak{S}_n$-invariant part in terms of generators as  

$$
H^2(S^n, \Z)^{\mathfrak{S}_n}=  \left\langle \sum_{i=1}^n \pr_i^*\alpha, \sum_{1\le i<j\le n} (\pr_i^*\beta\smile \pr_j^*\gamma - \pr_i^*\gamma\smile \pr^*_j\beta)\right\rangle_{\beta,\gamma \in H^1(S, \Z), \quad \alpha\in H^2(S, \Z).}$$

The equality $\pr_i\circ\Delta = \id$ gives
$$
\Delta^*\left(\sum_{i=1}^n\pr_i^*\alpha\right) = \sum_{i=1}^n \Delta^*\circ\pr_i^* \alpha =  n\,\alpha,
$$
and 
$$\Delta^*\left(\sum_{1\le i<j\le n} (\pr_i^*\beta\smile \pr_j^*\gamma - \pr_i^*\gamma\smile \pr^*_j\beta)\right) =2\sum_{1\le i<j\le n} \beta\smile \gamma = n(n-1)\,\beta\smile \gamma.$$

This shows that $\Delta^*(H^2(S^n, \Z)^{\mathfrak{S}_n})$ is equal to $nH^2(S, \Z)$, and 

that 
\[\bw2H^1(S,\Z) \lra H^2(S^n, \Z)^{\mathfrak{S}_n}\]
\[\beta\wedge \gamma \longmapsto (n-1)\sum_{i=1}^n \pr_i^* (\beta\smile \gamma) - \sum_{1\le i<j\le n} (\pr_i^*\beta\smile \pr_j^*\gamma -\pr_i^*\gamma\smile \pr_j^*\beta)\]
is an isomorphism onto \(\Ker \Delta^*\).
\end{proof}

The following Lemma will be used to compute \(\Br(E_n)\) in Theorem \ref{thm_brauer_computations}.

\begin{lem}\label[lem]{lem_E^40_S} Let \(S\) be a simply connected smooth projective surface and let \(G\) be a finite group acting freely on \(S\). 
For any positive integer \(n\) coprime to \(|G|\), we have \[E^{4,0}_4(S^{[n]}, G)=
0.\]
\end{lem}
\begin{proof}
Consider the diagonal embedding $\Delta : S \longrightarrow S^n$ and let $S^{(n)}$ be the $n$-th symmetric power of $S$, i.e., the quotient of $S^n$ by the symmetric group $\mathfrak{S}_n$. Finally, consider the Hilbert--Chow morphism
\(
h: S^{[n]} \to S^{(n)},
\)
which is a resolution of singularities of $S^{(n)}$ by \cite[Section 6]{Beauville83}.

As we illustrated in Section \ref{subsec_Carta_Leray}, the diagram of \(G\)-equivariant morphisms
\[
S^{[n]}\stackrel{h}{\lra} S^{(n)}\stackrel{p}{\longleftarrow} S^n\stackrel{\Delta}{\longleftarrow} S
\]
induces a diagram of spectral sequences
\begin{equation}\label{eqnss-ultima}
    E(S^{[n]}, G) \stackrel{h^*}{\longleftarrow} E(S^{(n)}, G) \xrightarrow{\Delta^*\circ p^*} E(S, G).
\end{equation}

Let us denote the differentials of these three spectral sequences  by \(d^{p,q}_r(S^{[n]}),d^{p,q}_r(S^{(n)})\) and \(d^{p,q}_r(S)\) respectively.

By hypothesis $S$ is simply connected, hence the same is true for $S^n$, and also for $S^{(n)}$ by \cite[Lemme 1 (b), Sec. 6]{Beauville83}; hence
\(E^{p,1}_2(X)=0\) for all \(p\) and any \(X\in \{S^{[n]},S^{(n)},S\}\). It follows that \(E_2^{1,2}(X)=E_3^{1,2}(X)\) and \(E_2^{4,0}(X)=E_3^{4,0}(X)\), which means the third page of \eqref{eqnss-ultima} contains a subdiagram of the form
\begin{equation}\label{eq_important_diagram}
\begin{tikzcd}[column sep = 50pt]
	{H^1(G,H^2(S^{[n]},\Z))} & {H^1(G,H^2(S^{(n)},\Z))} & {H^1(G,H^2(S,\Z))} \\
	{H^4(G,\Z)} & {H^4(G,\Z)} & {H^4(G,\Z)} ,    
	\arrow["{d^{1,2}_3(S^{[n]})}"', from=1-1, to=2-1]
	\arrow["{H^1(h^*)}"', from=1-2, to=1-1]
	\arrow["H^1(\Delta^* \circ p^*)", from=1-2, to=1-3]
	\arrow["{d^{1,2}_3(S^{(n)})}"', from=1-2, to=2-2]
	\arrow["{d^{1,2}_3(S)}", from=1-3, to=2-3]
	\arrow["\sim"', from=2-2, to=2-1]
	\arrow["\sim", from=2-2, to=2-3]
\end{tikzcd}\end{equation}
where lower horizontal maps are isomorphisms since the copies of \(\Z\) showing up as coefficients are \(H^0(S^{[n]},\Z),H^0(S^{(n)},\Z)\) and \(H^0(S,\Z)\) respectively and the pullback maps relating these groups are isomorphisms. 

Let us remark that the morphism \(d^{2,1}_3(S)\) is surjective: indeed, \[\Coker(d^{2,1}_3(S))=E^{4,0}_4(S,G)=E^{4,0}_\infty (S,G)\subset H^4(S/G,\Z)\simeq \Z\]
where the second equality is due to the fact \(H^3(S,\Z)=0\). In particular, $\coker(d^{2,1}_3(S))$ is a finite subgroup of $\Z$, so it must be $0$.

We now show that the morphism
\begin{equation}\label{eqn_morph_surj}
H^1(G,H^2(S^{(n)},\Z))\xrightarrow{H^1(\,\Delta^*\circ \,p^*)} H^1(G,H^2(S,\Z)),
\end{equation}
which is the top-right horizontal map in diagram~\eqref{eq_important_diagram}, is surjective.
By the short exact sequence \eqref{degree1_second_exact_seq_Br}, it follows that 
\[H^2(X,\Z)_{\tors}=0\] 
for any \(X\in \{S^{[n]},S^{(n)},S\}\), since $H^1(X,\Z/m\Z) = 0$ for every $m$.
Then the pullback map \(H^2(S^{(n)},\Z) \xrightarrow{p^*} H^2(S^n,\Z)^{\mathfrak{S}_n}\) is an injection of \(G\)-modules, since it becomes an isomorphism after tensoring by \(\C\) by  \cite[Lemma 2, Sec. 6]{Beauville83}. By \cite[Lemma 1.(2)]{Gugnin2012} and since \(\pi_1(S)=0\), it is also surjective: we sketch a proof based on classical work of Nakaoka for the reader's convenience.
Fix a point \(x_0\in S\) and denote by \(\iota \colon S\to S^{n}\) the map sending \(x\) to \((x, x_0,\dots, x_0)\). 
By \cite[Corollary 2.8]{Nakaoka57}, the map 
\[( p\circ \iota )^* \colon H^2(S^{(n)},\Z) \to H^2(S,\Z)\]
is surjective. 
Since $S$ is simply connected, the proof of \cref{lem_basta_svp} shows that any class \(\mu \in H^2(S^n, \Z)^{\mathfrak{S}_n}\) can be written as 
\[
\mu = \sum_{i=0}^n \pr_i\alpha
\]
for some \(\alpha \in H^2(S,\Z)\), and that $\Delta^*\mu = n\alpha$. 
Since $\pr_1 \circ \iota =\mathrm{Id}$  and \(\pr_i \circ \iota = x_0\) for $i\geq 2$, we have also $ \iota^* \mu = \alpha$. Therefore we have equalities \(\Delta^*  =n \ \iota^*,
\)
and  
\[\Delta^* \circ  p^* = n \ \iota ^* \circ  p^*= n ( p \circ \iota )^*.\]

Using again that $S$ is simply connected, by \cref{lem_basta_svp} the morphism \(\frac{\Delta^*}{n}\) is well-defined and an isomorphism of abelian groups. 
Hence, we can write \(p^*\) as
\[ p^* = \left(\frac{\Delta^*}{n}\right)^{-1} \circ ( p \circ \iota)^*.\]
This shows that the map
\begin{equation}\label{eq_sta_gran_bastarda}
    p^* \colon H^2(S^{(n)},\Z) \to H^2(S^n, \Z)^{\mathfrak{S}_n}
\end{equation} 
is an isomorphism.

The induced map in cohomology \(H^1(p^*)\) is also an isomorphism, so to prove that the morphism \(H^1(\Delta^* \circ p^*)\) of \eqref{eqn_morph_surj} is surjective we only need to show the surjectivity of
\begin{equation}\label{eq_map_1}
H^1\left(G,H^2(S^{n},\Z)^{\mathfrak{S}_n}\right)\xrightarrow{H^1(\Delta^*)} H^1(G,H^2(S,\Z)).\end{equation}

The sequence \eqref{eq_SES_1} is short exact since \(H^1(S,\Z)=0\), and taking its associated long exact sequence in group cohomology we have that 
    \[H^1(G,H^2(S^n,\Z)^{\mathfrak{S}_n}) \xrightarrow{H^1(G,\Delta^*)}H^1(G,H^2(S,\Z))\rightarrow H^1(G,H^2(S,\Z)\otimes \Z/n\Z)\]
    is exact.
    The right-most cohomology group is \(n\)-torsion since its coefficients are \(n\)-torsion. 
But \(H^1(G,Q)\) is always of \(|G|\)-torsion for any \(G\)-module \(Q\) (\cite[Theorem 6.5.8]{Weibel-intro_to_hom_algebra}), so if \(|G|\) and \(n\) are coprime we deduce that \(H^1(G,H^2(S,\Z)\otimes \Z/n\Z)=0\), hence (\ref{eq_map_1}) is surjective.

As we said, this implies that \(H^1(\Delta^* \circ p^*)\) is surjective. Looking at diagram~\eqref{eq_important_diagram} this implies \(d^{1,2}_3(S^{(n)})\) and \( d^{1,2}_3(S^{[n]})\) are surjective as well, which concludes the proof.
\end{proof}
\begin{rem}
    Note that in the proof of \cref{lem_E^40_S}, we do not need the action on $S^{[n]}$ and $S^{(n)}$ to be fixed point free. In fact, we never consider the abutment of the correspondent spectral sequences. This is crucial in Lemma \ref{lem_E^40_Kum}, where we will make a similar argument in a situation where the action is not free in general. 
\end{rem}

The two following lemmas deal with the case of Enriques manifolds whose universal cover is a generalized Kummer manifold.

\begin{lem}\label[lem]{lem_d_2^12(A[n+1])}
Let \(A\) be an abelian surface and \(G\) a finite cyclic group acting on \(A\). For any positive integer \(n\) such that \(|G|\) divides \(n\), we have
\[ d_2^{1,2}(A^{(n)})=0 
.\]
\end{lem}
\begin{proof}
    By the same reasoning as at the beginning of the proof of Lemma \ref{lem_E^40_S}, we can get a diagram just like \eqref{eqnss-ultima} with \(A\) in place of \(S\), from which we extract the subdiagram 
\[\begin{tikzcd}
	{E_2^{1,2}(A^{(n)})} && {E_2^{1,2}(A)} \\
	{E^{3,1}_2(A^{(n)})} && {E^{3,1}_2(A)},
	\arrow["{H^1(\Delta'^*)}", from=1-1, to=1-3]
	\arrow["{d_2^{1,2}(A^{(n)})}", from=1-1, to=2-1]
	\arrow["{d_2^{1,2}(A)}", from=1-3, to=2-3]
	\arrow["{H^3(\Delta'^*)}", from=2-1, to=2-3]
\end{tikzcd}\]
where we denoted by $\Delta'$ the composition $p\circ\Delta\colon A\to A^{(n)}$.

The bottom horizontal map is an isomorphism, since
\(\Delta'\) induces an isomorphism at the the level of the fundamental groups by \cite[Section 6, Lemme 1]{Beauville83}. Moreover, \[\Delta^* \colon H^2(A^{n},\Z)^{\mathfrak{S}_n} \lra H^2(A,\Z)\]
factors through multiplication by \(n\) by Lemma~\ref{lem_basta_svp}, hence so does the composition

 \[\Delta'^* \colon H^2(A^{(n)},\Z) \stackrel{p^*}{\lra} H^2(A^{n},\Z)^{\mathfrak{S}_n}\stackrel{\Delta^*}{\lra} H^2(A,\Z).\] 
Thus, \(H^1(\Delta'^*)\) factors through multiplication by \(n\) which, for groups \(H^1(G,M)\) with \(M\) any \(G\)-module is just the zero map. Indeed, by \cite[Theorem 6.5.8]{Weibel-intro_to_hom_algebra} positive degree cohomology of \(G\) is \(|G|\)-torsion and \(|G|\) divides \(n\). This proves \(H^1(\Delta'^*)=0\), which by the injectivity of \(H^3(\Delta'^*)\) implies that \(d_2^{1,2}(A^{(n)})\) is the zero map. 
\end{proof}

\color{black}
\begin{lem}\label[lem]{lem_E^40_Kum}
Let \(A\) be an abelian surface and \(G\) a finite cyclic group acting on \(A\). For any positive integer \(n\) such that \(|G|\) divides \(n+1\), the morphism
\(d_3^{1,2}(\Kum_n(A))\) is $0$.
\end{lem}
\begin{proof}
   Let $\theta$ be the inclusion $\Kum_n(A)\to A^{[n+1]}$. As in the proof of \cref{lem_E^40_S}, we have a morphism of spectral sequences
    $$
        E(\Kum_n(A), G) \xleftarrow{\theta^*\circ h^*} E( A^{(n+1)}, G) \xrightarrow{ \Delta'^*} E(A, G),
    $$
    where $h\colon A^{[n+1]} \to A^{(n+1)}$ is the Hilbert--Chow morphism, and $\Delta' \colon A\stackrel{\Delta}{\to} A^{n+1} \stackrel{p}{\to} A^{(n+1)}$ is the diagonal embedding.

First, we consider the following diagram, which relates the \(E_2^{4,0}\) and $E_3^{4,0}$ terms
\begin{equation}\label{eqnE40}
    \begin{tikzcd}
	{E^{4,0}_3(\Kum_n(A))} && {E^{4,0}_3(A^{(n+1)})} && {E^{4,0}_3(A)} \\
	{H^4(G,\Z)} && {H^4(G,\Z)} && {H^4(G,\Z)},
	\arrow[from=1-3, to=1-1]
	\arrow["{E^{4,0}_3(\Delta'^*)}", from=1-3, to=1-5]
    \arrow["{E^{4,0}_3(\theta^*\circ h^*)}"', from=1-3, to=1-1]
	\arrow[equal, from=2-1, to=1-1]
	\arrow[two heads, from=2-3, to=1-3]
	\arrow["\sim"', from=2-3, to=2-1]
	\arrow["\sim", from=2-3, to=2-5]
	\arrow[two heads, from=2-5, to=1-5]
    \end{tikzcd}
\end{equation}
where the lower row is just the terms \(E^{4,0}_2\) for all three varieties, the vertical maps are surjective by definition of \(E_3^{4,0}\), and equality on the left comes from the fact that $\Kum_n(A)$ is simply connected. Commutativity of the left square implies  $E_3^{4,0}(A^{(n+1)})=H^4(G, \Z)$, namely $d_2^{2,1}(A^{(n+1)}) = 0$. As in the proof of \cref{lem_d_2^12(A[n+1])}, we can consider the commutative square
\[\begin{tikzcd}
	{E_2^{2,1}(A^{(n+1)})} && {E_2^{2,1}(A)} \\
	{E^{4,0}_2(A^{(n+1)})} && {E^{4,0}_2(A)}.
	\arrow["{H^2(\Delta'^*)}", from=1-1, to=1-3]
	\arrow["{d_2^{2,1}(A^{(n+1)})}", from=1-1, to=2-1]
	\arrow["{d_2^{2,1}(A)}", from=1-3, to=2-3]
	\arrow["{\sim}", from=2-1, to=2-3]
\end{tikzcd}\]
Since the top horizontal arrow is an isomorphism by \cite[Section 6, Lemme 1]{Beauville83}, we get $d_2^{2,1}(A) = 0$. Hence, by looking at the right square in \eqref{eqnE40}, we obtain that the morphism $E_3^{4,0}(\Delta'^*)$ is an isomorphism.

 We now look at the commutative square induced by the differentials $d_3^{1,2}\colon E_3^{1,2}\to E_3^{4,0}$.
\[\begin{tikzcd}
	{H^1(G, H^2(\Kum_n(A), \Z))} && {H^1(G, H^2(A^{(n+1)}, \Z))} && {H^1(G, H^2(A, \Z))} \\
	{E^{1,2}_3(\Kum_n(A))} && {E_3^{1,2}(A^{(n+1)})} && {E_3^{1,2}(A)} \\
	{E^{4,0}_3(\Kum_n(A))} && {E^{4,0}_3(A^{(n+1)})} && {E^{4,0}_3(A)}. 
	\arrow["{H^1(\theta^*\circ h^*)}"', from=1-3, to=1-1]
	\arrow["{H^1(\Delta'^*)}", from=1-3, to=1-5]
	\arrow[equal, from=2-1, to=1-1]
    \arrow[equal, from=2-3, to=1-3]
	\arrow["{d_3^{1,2}(\Kum_n(A))}", from=2-1, to=3-1]
	\arrow[ from=2-3, to=2-1]
	\arrow["E_3^{1,2}(\Delta'^*)", from=2-3, to=2-5]
	\arrow["{d_3^{1,2}(A^{(n+1)})}", from=2-3, to=3-3]
	\arrow[hook, from=2-5, to=1-5]
	\arrow["{d_3^{1,2}(A)}", from=2-5, to=3-5]
	\arrow[, from=3-3, to=3-1]
	\arrow["{E^{4,0}_3(\Delta'^*)}", "\sim"', from=3-3, to=3-5]
\end{tikzcd}
\]
where the equality in the first column comes from the fact that $\Kum_n(A)$ is simply connected,  $E_3^{1,2}(A^{(n+1)})=E_2^{1,2}(A^{(n+1)})$ follows from \cref{lem_d_2^12(A[n+1])}, and the injectivity of the map on the right is by definition of $E_3^{1,2}$.

We established in the proof of \cref{lem_d_2^12(A[n+1])} that $H^1(\Delta'^*)$ is $0$. Therefore, the commutativity of the upper-right square implies that the map $E_3^{1,2}(\Delta'^*)$ is also $0$, and, by looking at the bottom-right square, we obtain that $d_3^{1,2}(A^{(n+1)})$ is $0$.

We claim that the morphism $H^1(\theta^*\circ h^*)$ is surjective. Using the commutativity of the left side of the above diagram, this implies that $d_3^{1,2}(\Kum_n(A))$ is $0$, and thus the lemma. 

Hence, we are left with proving the claim. By the discussion after \cite[Notation~5.13]{KapferMenet2018}, there exists a short exact sequence
    \begin{equation*}
         0 \lra H^2(A, \Z) \stackrel{j}{\lra} H^2(\Kum_n(A), \Z) \lra \Z\cdot\theta^*(\delta) \lra 0,
    \end{equation*}
 where $2\delta$ is the class of the exceptional divisor of the Hilbert-Chow morphism $h$. Its associated long exact sequence in group cohomology  reads
\begin{equation}\label{eq_rando1}
    \begin{tikzcd}[column sep = 50pt]
	{H^1(G,H^2(A,\Z))} & {H^1(G,H^2(\Kum_n(A),\Z))} & {H^1(G,\Z\cdot\theta^*(\delta))}.
	\arrow[from=1-1, to=1-2, "{H^1(j^*)}"]
	\arrow[from=1-2, to=1-3]
\end{tikzcd}
\end{equation}
Since \(G\) acts trivially on \(\theta^*(\delta)\), 
we have
\(H^1(G,\Z\cdot \theta^*(\delta))=0\) by \cref{lem_cyclic_group_cohom}, and by 
exactness of \cref{eq_rando1} it follows that the map $H^1(j^*)$ is surjective.  

Finally, by \cite[Equation (19)]{KapferMenet2018}, the morphism $j$ factors  as
    \begin{equation}
        j\colon H^2(A, \Z) \stackrel{d}{\lra} H^2(A^{(n+1)}, \Z) \xrightarrow{\theta^*\circ h^*} H^2(\Kum_n(A), \Z),
    \end{equation}
    where the morphism $d$ sends a class $\alpha\in H^2(A, \Z)$ to 
    $$
       \alpha\otimes 1\otimes \cdots \otimes 1 + 1\otimes \alpha\otimes 1 \otimes \cdots \otimes 1 +\cdots + 1\otimes \cdots\otimes 1\otimes \alpha \in H^2(A^{(n+1)}, \Q)=H^2(A^{n+1}, \Q)^{\mathfrak{S}_{n+1}}, $$
    which is actually in \(H^2(A^{(n+1)}, \Z)\) by  \cite[Lemma~1.(2)]{Gugnin2012}.
    The surjectivity of  $H^1(j^*)$ implies that $H^1(\theta^*\circ h^*)$ is surjective as well. This proves the claim, and concludes the
    proof.
\end{proof}

\section{Brauer groups of \texorpdfstring{\(E_n\)}{En} and \texorpdfstring{\(T_n\)}{Tn}}\label[section]{sec:examples}


We apply the results of \cref{secBrEnrM} to the  Enriques manifolds introduced in \cref{subsec_known_constructions}.

\begin{thm}\label[thm]{thm_brauer_computations}
Consider the Enriques manifolds of type \(E_n\) and \(T_{3m-1}\) with $n$ and $m$ odd integers, constructed in \cref{subsec_known_constructions} as \ref{item_E_n} and \ref{item_T_n} respectively. One has that
    \[\Br(E_n)=\Z/2\Z.\]
    \[\Br(T_{3m-1})=(\Z/3\Z)^{\oplus 2}.\]
\end{thm}
\begin{proof}
Recall that $E_n$ is obtained as a quotient of $S^{[n]}$ by $i^{[n]}$, where $S$ is a K3 surface with a fixed point free involution $i$, and $T_{3m-1}$ is the quotient of $\Kum_{3m-1}(C\times C')$ by $\psi^{\llbracket 3m-1\rrbracket}$, where $C$ and $C'$ are  elliptic curves, and $\psi$ a fixed point free order $3$ automorphism of $C\times C'$. 
The groups \(H^3(S^{[n]},\Z)\) and \(H^3(\Kum_{3m-1}(C\times C'), \Z)\) have no torsion by \cref{lem_torsion_free_Hilbert} and \cite[Corollary 6.3]{KapferMenet2018} for $m=1$ and by \cite[Corollary 1.3]{Hartlieb_Verni25} for odd $m\geq 3$. Therefore, \cref{prop_sequences_for_AMinv} and \cref{eq_third_exact_se_Br} shows that the Brauer groups of $E_n$ and $T_{3m-1}$ can be computed using the exact sequence \eqref{eq_Artin_Mumford_SES}.

We start with the manifold $E_n$. Since $n$ is odd and the group  \(G=\langle i^{[n]}\rangle\) has order $2$, \cref{lem_E^40_S} shows that the exact sequence \eqref{eq_Artin_Mumford_SES} simplifies to 
the following short exact sequence of \(\Z/2\Z\)-vector spaces:
\begin{equation}\label{eq_Br_of_S^n/i^n}
\begin{tikzcd}
	0 & {H^3(E_n,\Z)_{\mathrm{tors}}} & {H^1(G,H^2(S^{[n]},\Z))} & {\Z/2\Z} & 0
	\arrow[from=1-1, to=1-2]
	\arrow[from=1-2, to=1-3]
	\arrow["{d^{1,2}_3}", from=1-3, to=1-4]
	\arrow[from=1-4, to=1-5]
\end{tikzcd}.\end{equation}
The middle term can be computed as follows: consider the isomorphism
\[H^2(S^{[n]},\Z)\simeq H^2(S,\Z)\oplus \Z \delta, \]
where \(i^{[n]*}\) acts as \(i^*\) on the component \(H^2(S,\Z)\) and trivially on the second term. By \cite[Section 5.1]{Bea09-Brauer_of_Eriques_surf}), the lattice \(H^2(S,\Z)\) with the Beauville--Bogomolov--Fujiki form decomposes as 
\[H^2(S,\Z)\simeq F\oplus F \oplus U,\]
where \(F \coloneqq  E_8(-1) \oplus U\) is of rank \(10\) and \(U\) is a copy of the hyperbolic lattice; and the action of \(i^*\) is given by
\[(\alpha, \alpha',\beta)\longmapsto (\alpha ', \alpha ,-\beta).\]
Therefore, we can compute
\[\Ker (\id + i^{[n]*})=\Ker (\id + i^{*})=\langle(v,-v)\in F\oplus F\,|\,v\in F\rangle\oplus U\]
\[\Im (\id - i^{[n]*})=\Im (\id - i^{*})=\langle(v,-v)\in F\oplus F\,|\,v\in F\rangle\oplus\, 2\,U,\]
and \cref{lem_cyclic_group_cohom} implies \[H^1(G,H^2(S^{[n]},\Z))=U/2U\simeq (\Z/2\Z)^2.\]
The short exact sequence~\eqref{eq_Br_of_S^n/i^n} gives \[\Br(E_n)\simeq \Z/2\Z.\]

We move to the case of $T_{3m-1}$. In this case, using \cref{lem_E^40_Kum},
the short exact sequence~\eqref{eq_Artin_Mumford_SES} yields an isomorphism 
\[H^3(T_{3m-1},\Z)_\mathrm{tors}\simeq H^1\left(\Z/3\Z, H^2\left(\mathrm{Kum}_{3m-1}\left(C\times C'\right),\Z\right)\right).\]
To compute this group cohomology, we apply \cref{lem_cyclic_group_cohom} to the action of $\psi^{\llbracket {3m-1}\rrbracket}$ on \[H^2\left(\mathrm{Kum}_{3m-1}\left(C\times C'\right)\right) \simeq H^2(C\times C',\Z) \oplus \Z\delta.\] 
Since the action comes from the abelian surface, $\psi^{\llbracket {3m-1}\rrbracket}$ fixes the exceptional class $\delta$. We first determine the kernel of $N\coloneqq  1 + \psi + \psi^2$. The cohomology of $C\times C'$ decomposes as
\[
H^2(C\times C',\Z) \simeq H^2(C,\Z) \oplus \left(H^1(C,\Z)\otimes H^1(C',\Z)\right) \oplus  H^2(C',\Z) 
\]

Note that the action of $\psi$ on $H^2(C,\Z) $ and $H^2(C',\Z) $ is trivial, hence we are left to consider its restriction to $H^1(C,\Z)\otimes H^1(C',\Z)$. In our hypothesis, we can assume that the elliptic curve $C$ is given by the rank two lattice 
\[
\langle e,f\rangle \coloneqq  \langle \begin{pmatrix}1\\0 \end{pmatrix},\begin{pmatrix}-\frac{1}{2}\\\frac{\sqrt{3}}{2} \end{pmatrix}\rangle \subseteq \C,
\]
and that the linear automorphism $h_3$ of $C$ of order \(3\) is given by the multiplication by $\zeta_3$.
Therefore, the action of $\psi$ on $H^1(C,\Z)\otimes H^1(C',\Z)$ is given by
\begin{align*}
    de_1 \otimes de_2 &\longmapsto df_1\otimes de_2,\\
    df_1 \otimes de_2 &\longmapsto (-de_1-df_1)\otimes de_2,\\
    de_1 \otimes df_2 &\longmapsto df_1\otimes df_2,\\
    df_1 \otimes df_2 &\longmapsto (-de_1-df_1)\otimes df_2,
\end{align*}
where $\{de_i,df_i\}$ for $i = 1,2$ denotes the generators of $H^1(C,\Z)$ and $H^1(C',\Z)$ respectively.
An explicit computation shows that the restriction of $N$ to $ H^1(C,\Z)\otimes H^1(C',\Z)$ is trivial. Therefore, $\mathrm{Ker}(N)$ is equal to $H^1(C,\Z)\otimes H^1(C',\Z)$. We are left to compute the image of $1-h_3$, whose matrix is 
\[
\begin{pmatrix}
1 & 1 & 0 & 0 \\
-1 & 2 & 0 & 0 \\
0 & 0 & 1 & 1 \\
0 & 0 & -1 & 2 
\end{pmatrix}
\]
The diagonal of its Smith normal form is $(1,1,3,3)$. In particular, 
\[
\faktor{\mathrm{Ker}(N)}{\mathrm{Im}(1-h_3)} \simeq (\Z/3\Z)^{\oplus 2}.
\]
This shows that $\Br(T_{3m-1})=(\Z/3\Z)^{\oplus 2}$ and concludes the proof.
\end{proof}

\begin{rem}\label[rem]{remark}
The proof of \cref{thm_brauer_computations} actually determines the generator of the Brauer group of \(E_n\). Indeed, 
\[\Br(E_n)=\Ker d_3^{1,2}(S^{[n]}, \langle i^{[n]}\rangle).\]
Let $E\coloneqq \faktor{S}{\langle i\rangle}$ and denote by \(b_E\) the generator of \(\Br(E)=\Ker d_3^{1,2}(S, \langle i\rangle)\). Since the kernel of the right and left vertical maps in diagram~(\ref{eq_important_diagram}) are isomorphic, we see that if \(\overline{b}\in H^1(G,H^2(S^{(n)},\Z))\) is a preimage of \(b_E\) along \(H^1(\Delta^* \circ p^*)\), then \[\Br(E_n)=\{0, H^1(h^*) (\overline{b})\}.\]
\end{rem}

\begin{rem}\label[remark]{rem:RnKn}
    If one can prove that $H^3(\Kum_{odd}(A),\Z)_{\tors}=0$, then the same argument exhibited to compute $\mathrm{Br}(T_{3m-1})$ applies also to all the examples $K_n$, $T_n$ and $R_n$. Take for example $R_{n}$ with $n = 4 m-1$. By \cref{lem_E^40_Kum}, one still get that 
    \[H^3(R_n,\Z)_\mathrm{tors}\simeq H^1\left(\Z/4\Z, H^2\left(\mathrm{Kum}_n\left(C\times C'\right)\right)\right).\]
    In this case, we can assume that the elliptic curve $C$ is given by the rank two lattice 
    \[
    \langle e,f\rangle \coloneqq  \langle \begin{pmatrix}1\\0 \end{pmatrix},\begin{pmatrix}0\\1 \end{pmatrix}\rangle \subseteq \C,
    \]
    and the action of $h_4$ on the first factor is given by the multiplication by $\zeta_4= \sqrt{-1}$. Again the kernel of $N$ is equal $H^1(C, \Z)\otimes H^1(C', \Z)$, and the action of $\Id-\psi$ on $H^1(C, \Z)\otimes H^1(C', \Z)$ is given in the usual basis by the matrix
\[
\begin{pmatrix}
1 & 1 & 0 & 0 \\
-1 & 1 & 0 & 0 \\
0 & 0 & 1 & 1 \\
0 & 0 & -1 & 1
\end{pmatrix}
\]
whose Smith normal form has diagonal $(1,1,2,2)$. In particular, 
\[
 H^1\left(\Z/4\Z, H^2\left(\mathrm{Kum}_n\left(C\times C'\right)\right)\right) =\faktor{\mathrm{Ker}(N)}{\mathrm{Im}(1-h_4)} \simeq (\Z/2\Z)^{\oplus 2}
\]
Therefore, if $H^3(\Kum_n(A),\Z)_{\tors}=0$ then $\mathrm{Br}(R_n) = (\Z/2\Z)^{\oplus 2}$.

A similar computation shows that, if \(H^3(\Kum_{n}(A),\Z)_{\tors}=0\) for some odd positive integer~\(n\), then the Brauer group of $K_n$ is equal to%
\[\Br(K_n)=H^1\left(\langle \iota^{\llbracket n \rrbracket} \rangle, H^2(\Kum_n(A),\Z)\right)=(\Z/2\Z)^{\oplus 4}.\]
\end{rem}

We end this section with a result on the integral cohomology of $S^{(n)}$ which will be useful in \cref{secVanBC}. The idea of the proof is in the same spirit of the previous ones. Before that, we need the following lemma.

\begin{lem}\label[lem]{torsH3}
Let $X$ be a connected simply connected projective variety such that $H^3(X,\Z)$ is torsion free. For every $n \geq 1$, $H^3(X^{(n)},\Z)_{\tors}= 0$.
\end{lem}
\begin{proof}
By the Universal coefficients theorem, it is enough to check that $H_2(X^{(n)},\Z)$ is torsion free. By a result of Steenrod (see e.g. \cite{dold}[Theorem 2]), we have that the homology of $X^{(n)}$ is a direct summand of the homology of the infinite symmetric power $X^{(\infty)}$. The latter can be defined as the colimit of the chain of maps
\[
X \longrightarrow X^{(2)}\longrightarrow X^{(3)} \longrightarrow \dots
\]
given by the sum of a fixed point $x_0\in X$. Therefore, it suffices to prove that $H_2(X^{(\infty)},\Z)$ is torsion free. By the Dold--Thom theorem 
(see \cite[Corollary 4.K.7]{Hatcher}), \(X^{(\infty)}\) has the same homotopy type as the product of the Eilenberg--MacLane spaces of the homology groups of $X$, which implies
\[\pi_i(X^{(\infty)})\simeq H_i(X,\Z).\]
By Hurewicz's Theorem and the hypothesis \(\pi_1(X)=0\), we have \[H_2(X^{(\infty)},\Z)=\pi_2(X^{(\infty)})\simeq H_2(X,\Z).\]
Again by the Universal coefficients theorem, the hypothesis $H^3(X,\Z)_{\tors}=0$ implies that $ H_2(X,\Z)_{\tors}=0$. The proof is complete.
\end{proof}

\begin{prop}\label[prop]{torsSn}
Let \(S\) be a K3 surface with a fixed point-free anti-symplectic involution \(i\colon S \rightarrow S\) and consider the induced involution $i^{(n)}$ on $S^{(n)}$ for $n$ odd. Let $E_{(n)} := \faktor{S^{(n)}}{\langle i^{(n)}\rangle}$. There is an isomorphism 
\[
H^3\left(E_{(n)},\Z\right)_{\tors}\simeq \Z/2\Z.
\]
\end{prop}
\begin{proof}
By \cref{torsH3}, \(H^3(S^{(n)},\Z)_{\tors}=0\). Since the action of $i^{(n)}$ is free for $n$ odd, we can use the Cartan--Leray spectral sequence associated to $G = \langle i^{(n)}\rangle$ to obtain an analogue of the exact sequence (\ref{eq_Artin_Mumford_SES}) for $X = S^{(n)}$. Note that the second row of the second page of the spectral sequence is trivial as $S^{(n)}$ is simply connected. 
The proof of \cref{lem_E^40_S} shows that $d_3^{1,2}(S^{(n)})$ is surjective hence the exact sequence~\eqref{eq_Artin_Mumford_SES} simplifies to:
\begin{equation}\label{eq_Br_of_S^(n)/i^(n)}
\begin{tikzcd}
	0 & {H^3\left(E_{(n)},\Z\right)_{\tors}} & {H^1(G,H^2(S^{(n)},\Z))} & {\Z/2\Z} & 0.
	\arrow[from=1-1, to=1-2]
	\arrow[from=1-2, to=1-3]
	\arrow["{d^{1,2}_3}", from=1-3, to=1-4]
	\arrow[from=1-4, to=1-5]
\end{tikzcd}\end{equation}
Moreover, we proved that there is $\Z/2\Z$-equivariant isomorphism 
$$H^2(S^{(n)}, \Z) \simeq H^2(S^n,\Z)^{\mathfrak{S}_n} \xrightarrow{\frac{1}n \Delta^*} H^2(S,\Z).$$ 
The same computations as in \cref{thm_brauer_computations} show
\[H^1(\Z/2\Z,H^2(S^{(n)},\Z)) =U/2U\simeq (\Z/2\Z)^{\oplus 2},\]
hence from the short exact sequence~(\ref{eq_Br_of_S^(n)/i^(n)}) we obtain the isomorphism \[H^3\left(E_{(n)},\Z\right)_{\tors}\simeq \Z/2\Z.\]
\end{proof}

\section{Vanishing of the Brauer class}\label[section]{secVanBC}
Let $S$ be a K3 surface with a fixed point free anti-symplectic involution \(i\colon S \rightarrow S\) and denote by $E$ the associated Enriques surface. For any odd $n$, we denote by 
$\sigma$ the induced involution $i^{[n]}$ on $S^{[n]}$, by $E_n$ the Enriques manifold $S^{[n]}/\langle\sigma\rangle$ and by $\pi\colon S^{[n]}\to E_n$ the quotient map. Let $b_n$ be the generator of $\Br(E_n)\simeq \Z/2\Z$ (see \cref{thm_brauer_computations}). 
Following the steps of \cite{Bea09-Brauer_of_Eriques_surf}, we study when the pullback morphism
$$
\pi^*\colon \Br(E_n) \lra \Br(S^{[n]})
$$
is zero, or equivalently when $\pi^*(b_n) = 0$.

To generalize the results in \cite{Bea09-Brauer_of_Eriques_surf}, we need to relate the cohomologies of the surface $S$ and of the manifold $S^{[n]}$. This is done by passing through the quotient $E_{(n)}$ of $S^{(n)}$ under the involution $i^{(n)}$. More precisely, we have the following diagram:
\begin{equation}\label{ch6_comm_squares-new}
    \begin{tikzcd}
    {S^{[n]}} \arrow[r, "h"] \arrow[d, "\pi"'] & S^{(n)} \arrow[d, "\pi^{(n)}"] & S \arrow[l, "\Delta'"'] \arrow[d, "\pi_E"] \\
    E_n \arrow[r, "f"'] & E_{(n)}  & E, \arrow[l, "g"]   
    \end{tikzcd}
    \end{equation}
    where both the maps $f$ and $g$ are defined thanks to the fact that the action of the symmetric group commutes with the action of the involution, and $\Delta' := p \circ \Delta$ as in \cref{secBrEnrM}.
 
In the following, for any normal variety $X$, we denote by $K_X\in H^2(X,\Z)$ the first Chern class of the canonical bundle, and by $k_X$ the image of $K_X\otimes 1$ in $H^2(X, \Z/2\Z)$ (see~\eqref{first_exact_seq_Br} with $n=2$).
\begin{lem}\label[lem]{lem:g*}
    The morphism
    \[
        g^* :H^2\left(E_{(n)},\Z \right) \longrightarrow H^2(E,\Z).
    \]
    is injective, and has image $nH^2(E,\Z)$. As a consequence, the morphism
    \[
        g^*_2 :H^2\left(E_{(n)},\Z/2\Z \right) \longrightarrow H^2(E,\Z/2\Z).
    \]
    is an isomorphism sending the class $k_{E_{(n)}}$ to $k_E$.
\end{lem}
\begin{proof}
By the functoriality of the Universal coefficients theorem (which follows from the fact it is an application of the Grothendieck spectral sequence) we have the following commutative diagram
\[\begin{tikzcd}
	0 & {\Ext(H_1(E_{(n)},\Z),\Z)} & {H^2(E_{(n)},\Z)} \\
	0 & {\Ext(H_1(E,\Z),\Z)} & {H^2(E,\Z)}
	\arrow[from=1-1, to=1-2]
	\arrow[from=1-2, to=1-3]
	\arrow[from=1-2, to=2-2]
	\arrow["{g^*}", from=1-3, to=2-3]
	\arrow[from=2-1, to=2-2]
	\arrow[from=2-2, to=2-3].
\end{tikzcd}\]
If \(\gamma\) is the non-trivial element of \(H_1(E)\), then for any \(x\in S\) one has
\[\Delta'(\gamma \cdot x)=g_*\gamma \cdot \Delta'(x)\]
by equivariance of \(\Delta'\) under monodromy action. Since the monodromy action of \(\gamma\) is non-trivial and \(\Delta'\) is injective, this means the monodromy action of \(g_*\gamma\) is also non-trivial, hence \(g_* \colon H_1(E,\Z) \to H_1(E_{(n)},\Z)\) is an isomorphism. This implies the left vertical arrow of the above diagram is an isomorphism. 
Since the image of the top and bottom horizontal maps are the classes \([K_{E_{(n)}}]\) and \([K_E]\) respectively (these being the only non-trivial torsion classes in the respective groups), we conclude that \(g^*[K_{E_{(n)}}]=[K_E]\).

Then, the Cartan--Leray spectral sequence $E(-,\langle \sigma \rangle)$ with $\Z$-coefficients for the maps 
$\pi^{(n)}$ and $\pi_E$ gives the the following commutative diagram (see \cref{secBrEnrM} for details)
\begin{equation}\label{diagram_spec_seq_Z-new}
\begin{tikzcd}
0 \arrow[r] & \langle K_{E_{(n)}} \rangle \arrow[r]  \arrow[d] & {H^2\left(E_{(n)},\Z \right)} \arrow[r, "\pi^{(n)^*}"] \arrow[d, "g^*"] & {H^2(S^{(n)},\Z )^{\sigma^* }} \arrow[d, "\Delta'^*", hook]  \arrow[r] & 0\\
0 \arrow[r] & \langle K_E \rangle \arrow[r] & {H^2(E,\Z )} \arrow[r, "\pi_E^*"] & {H^2(S,\Z )^{\sigma^*}} \arrow[r] & 0.
\end{tikzcd}
\end{equation}

The first result follows from the fact, established in \cref{lem_basta_svp}, that the morphism \(\Delta'^* \colon H^2(S^{(n)},\Z) \to H^2(S,\Z)\) is injective and has image \(n H^2(S,\Z)\), together with the remark that \[K_E=nK_E \in nH^2(E,\Z)\] since $n$ is odd and \(2K_E =0\).

For the second statement, consider the Cartan--Leray spectral sequence $E(-,\langle \sigma \rangle)$ with $\Z/2\Z$-coefficients for the maps $\pi^{(n)}$ and $\pi_E$. It gives the following commutative diagram
\begin{equation}\label{diagram_spec_seq_Z2}
\begin{tikzcd}
0 \arrow[r] & \langle k \rangle \arrow[r]  \arrow[d] & {H^2\left(E_{(n)},\Z/2\Z\right)} \arrow[r, "\pi^{(n)^*}"] \arrow[d, "g^*_2"'] & {H^2(S^{(n)},\Z/2\Z)^{\sigma^*}}  \arrow[d, "\Delta'^*_2"', "\vsim"]  \\
0 \arrow[r] & \langle k_E \rangle \arrow[r] & {H^2(E,\Z/2\Z)} \arrow[r, "\pi_E^*"] & {H^2(S,\Z/2\Z)^{\sigma^*}}.
\end{tikzcd}
\end{equation}
By the short exact sequence \eqref{second_exact_seq_Br},
\(H^2(S,\Z/2\Z)= H^2(S,\Z)\otimes \Z/2\Z\), and the same holds for \(S^{(n)}\) because the torsion of its third integral cohomology vanishes by \cref{torsH3}. 
Then the right most vertical map is an isomorphism since it coincides with the reduction modulo \(2\) of the pullback with integral coefficients, and $n$ is odd. 


Note that the map \(\pi^{(n)^*}\) is surjective, since the subspace \(H^2(E_{(n)}, \Z)\otimes \Z/2\Z\) surjects onto \(H^2(S^{(n)},\Z)^{\sigma^*}\otimes \Z/2\Z =H^2(S^{(n)},\Z/2\Z )^{\sigma^*}\) by the exactness of the first row of \eqref{diagram_spec_seq_Z-new}. One then concludes by Snake Lemma.
\end{proof}
We recall here some facts about the second cohomology groups of $S^{[n]}$ and $E_n$. 
Denote by $F$ the abstract even lattice $E_8(-1) \oplus U$, and let $q$ be the Beauville--Bogomolov--Fujiki form on $S^{[n]}$.
As in the proof of \cref{thm_brauer_computations}, recall that there is an isometry
\begin{equation}\label{eqn:latticeSn}
    \left(H^2(S^{[n]},\Z),q\right)\simeq F\overset{\perp}{\oplus} F \overset{\perp}{\oplus}(-2(n-1)) \overset{\perp}{\oplus} U,
\end{equation}
under which the action of $\sigma$ on $H^2(S^{[n]},\Z)$ is given by
\begin{equation}\label{eqn:involution}
 \sigma^*:(a,a',m,b) \longmapsto (a',a,m,-b).
\end{equation}
Moreover, the Cartan--Leray spectral sequence for the morphism $\pi\colon S^{[n]}\to E_n$ yields \begin{equation}\label{eqn:SecondCohoEn}
0 \lra \Z/2\Z \lra H^2(E_n, \Z) \stackrel{\pi^*}{\lra}  H^2(S^{[n]}, \Z)^{\sigma^*} \lra 0,
\end{equation}
where $\ker(\pi^*)$ is generated by the canonical bundle $K_{E_n}$ and the elements of \(H^2(S^{[n]}, \Z)^{\sigma^*}\)are those of the form \((a,a,m,0)\).
\begin{rem}\label[rem]{lattices}
For any \emph{even} lattice $(L,q)$ we define a \(\Z/2\Z\)-vector space with quadratic form by setting
 $$(L_2,\,\Tilde{q}) \coloneqq  \left(L \otimes \Z/2, \,\,\frac{1}{2}\,q\otimes \id\right).$$
Using the short exact sequence \eqref{second_exact_seq_Br}, $H^2(S^{[n]},\Z)\otimes \Z/2$ is isomorphic to $H^2\left(S^{[n]},\Z/2\right)$.  From ~\eqref{eqn:latticeSn} we get an isometry
\[
\left(H^2(S^{[n]},\Z/2), \tilde q\right)\simeq F_2\overset{\perp}{\oplus} F_2\overset{\perp}{\oplus} (0) \overset{\perp}{\oplus} U_2.
\] 
In particular, $\tilde q$ is a degenerate quadratic form with kernel $\{0,\delta\}$, where we still denote by $\delta$ the generator of the lattice $(0)$.
\end{rem}
Finally, the long exact sequence~\eqref{first_exact_seq_Br} with $n=2$ yields, using $\mathrm{Pic}(E_n) = H^2(E_n,\Z)$,
\begin{equation}\label{seqF2-new}
    0\longrightarrow H^2(E_n,\Z) \otimes \Z/2\Z \longrightarrow H^2\left(E_n, \Z/2\Z \right)\longrightarrow \mathrm{Br}(E_n)[2] \longrightarrow 0,
\end{equation}
where $H^2(E_n,\Z) \otimes \Z/2\Z$ is a $\Z/2\Z$-vector space of dimension $12$.\\

Let $\Delta_{F_2}\subset F_2\oplus F_2$ be the diagonal embedding of \(F_2\) and let $\epsilon$ be the unique element of $U_2$ with square equal to $1$.
The following proposition is the analogue of \cite[Proposition~5.3]{Bea09-Brauer_of_Eriques_surf} for $E_n$.
\begin{prop}\label[prop]{kerImPullB}
    The pull-back $\pi^*_2:H^2\left(E_n,\Z/2\Z\right)\longrightarrow H^2\left(S^{[n]},\Z/2\Z\right)$ satisfies:
    \begin{enumerate}[label = \rm(\roman*)]
        \item\label{itemKer} $\mathrm{Ker}(\pi^*_2) = \{0,k_{E_n}\}$;
        \item\label{itemIm} $\mathrm{Im}(\pi^*_2) = \Delta_{F_2} \oplus \Z/2\Z\cdot\delta \oplus \Z/2\Z\cdot \epsilon$.
    \end{enumerate}
\end{prop}
\begin{proof}
    Consider the Cartan--Leray spectral sequence $E(S^{[n]},\langle \sigma\rangle)$ with $\Z/2$-coefficients. The third page
\begin{equation*}
    \begin{tikzcd}
        \textcolor{gray}{q = 2}&  H^2(S^{[n]}, \Z/2)^{\sigma^*} \ar[rrrdd, "d_3^{0,2}"]& * & * & * \\
        \textcolor{gray}{q = 1}&  0 & 0 & 0 & 0 & \\
        \textcolor{gray}{q = 0}&  \Z/2& \Z/2 & \Z/2 & \Z/2  \\
        &\textcolor{gray}{p=0} &\textcolor{gray}{p=1} & \textcolor{gray}{p=2} & \textcolor{gray}{p=3}\\
    \end{tikzcd}
\end{equation*}
gives the short exact sequence
\begin{equation}\label{Spec_seq_T-new}
\begin{tikzcd}
	0 & {\Z/2} & {H^2(E_n,\Z/2)} & {\mathrm{Ker}\left(d_3^{0,2}\right)} & 0.
	\arrow[from=1-1, to=1-2]
	\arrow[from=1-2, to=1-3]
	\arrow["{\pi^*_2}", from=1-3, to=1-4]
	\arrow[from=1-4, to=1-5]
\end{tikzcd}
\end{equation}
Hence, $\Ker(\pi^*_2) \simeq \Z/2$. Since $\pi^*{K_{E_n}} \simeq \mathcal{O}_{S^{[n]}}$, \cref{itemKer} follows. 

We now prove \cref{itemIm}. A dimensional count in the exact sequence \eqref{seqF2-new} shows that the $\Z/2\Z$-vector space  $H^2\left(E_n,\Z/2\Z\right)$ has dimension $13$; therefore, using the exact sequence~\eqref{Spec_seq_T-new}, we obtain that $\mathrm{Im}(\pi^*_2)=\ker(d_3^{0,2})$ has dimension $12$.  

By the descriptions of \(H^2(S^{[n]},\Z/2\Z)\) given in \cref{lattices} and of the involution $\sigma^*$ given in \eqref{eqn:involution}, we obtain that the $\sigma^*$-invariant of \(H^2(S^{[n]},\Z/2\Z)\) is 
$$
H^2(S^{[n]},\Z/2\Z)^{\sigma^*} = \Delta_{F_2}\oplus \Z/2\Z\cdot \delta\oplus U_2.
$$

One then must have \(\pi^*_2(H^2(E_n, \Z/2\Z))\subset \Delta_{F_2}\oplus \Z/2\Z\cdot \delta\oplus U_2.\) By \eqref{eqn:SecondCohoEn} we also have
\[\Delta_{F_2}\oplus \Z/2\Z \cdot \delta =\pi^*(H^2(E_n, \Z))\otimes \Z/2\Z\subset \pi^*_2(H^2(E_n, \Z/2\Z)).\]
To conclude the proof we need to show that $\epsilon\in U_2$ lies in $\pi^*_2(H^2(E_n, \Z/2\Z))$.

Consider the commutative diagram 
\[
\begin{tikzcd}[column sep = 20pt]
0 \arrow[r] & {H^2(E_n,\Z)\otimes \Z/2\Z} \arrow[r] & {H^2(E_n,\Z/2\Z)}\arrow[r] & \mathrm{Br}(E_n) \arrow[r] & 0 \\
0 \arrow[r] & {H^2\left(E_{(n)},\Z\right)\otimes \Z/2\Z} \arrow[r] \arrow[u, "f^*\otimes \id"] \arrow[d, "g^*\otimes \id"] & {H^2\left(E_{(n)},\Z/2\Z\right)} \arrow[r] \arrow[u, "f_2^*"] \arrow[d, "g_2^*"', "\vsim"] & H^3\left(E_{(n)},\Z\right)[2] \arrow[r] \arrow[u, "f_{\Br}^*", "\vsim"'] \arrow[d, "g_{\Br}^*"', "\vsim"] & 0 \\
0 \arrow[r] & {H^2(E,\Z)\otimes \Z/2\Z} \arrow[r] & {H^2(E,\Z/2\Z)}\arrow[r] & \mathrm{Br}(E) \arrow[r] & 0,
\end{tikzcd}
\]
where the short exact sequences come from applying \eqref{second_exact_seq_Br} to $E_n$, 
$E_{(n)}$ 
and $E$ respectively, while the vertical arrows arise from diagram~\eqref{ch6_comm_squares-new}. The morphism $g_2^*$ is an isomorphism by \cref{lem:g*}, and $f^*_{\Br}\colon 
H^3(E_{(n)},\Z)[2]\to \Br(E_n)$
is an isomorphism by \cref{torsSn,remark}.

Take a class $\beta_E \in H^2(E,\Z/2\Z)$ which does not come from $H^2(E,\Z)\otimes \Z/2\Z$. By \cite[(3.2)]{Bea09-Brauer_of_Eriques_surf} (see also the proof of Proposition 5.3), the self-intersection of $\beta_E$ is $1$, and by \cite[Lemma 5.4]{Bea09-Brauer_of_Eriques_surf}, $\Tilde{q_S}(\pi_{E,2}^*\beta_E)=1 $, where $q_S$ is the intersection form on the K3 surface $S$. 
The class $$\beta_n \coloneqq  (f_2^*\circ(g_2^*)^{-1})(\beta_E)\in H^2(E_n,\Z/2\Z)$$ does not vanish in $\mathrm{Br}(E_n)$, by diagram chasing. 

Since the morphisms $\Delta'$ and $h$ of diagram~\eqref{ch6_comm_squares-new} are equivariant with respect to the $\Z/2\Z$-actions, we obtain a morphism
\begin{equation}
    \varphi:H^2(S,\Z/2\Z)^{\sigma^*} \xrightarrow{(\Delta'^*_2)^{-1}} H^2(S^{(n)},\Z/2\Z)^{\sigma^*}\xhookrightarrow{h^*_2}H^2(S^{[n]},\Z/2\Z)^{\sigma^*},
\end{equation}
preserving the quadratic forms (in the proof of \cref{lem:g*} we also showed that $\Delta'^*_2$ is an isomorphism).
Commutativity of diagram~\eqref{ch6_comm_squares-new} implies that $\pi^*_2\circ f^*_2\circ(g^*_2)^{-1} = \varphi \circ \pi_{E,2}^*$. 
Therefore, we can compute
\[
\widetilde{q_{S^{[n]}}}(\pi^*_2\beta_{n}) = \widetilde{q_{S^{[n]}}}\left((\pi^*_2\circ f^*_2\circ(g^*_2)^{-1})(\beta_E)\right) = \widetilde{q_{S^{[n]}}}\left((\varphi \circ \pi_{E,2}^*)(\beta_E)\right)= \widetilde{q_S}\left( \pi_{E,2}^*\beta_E\right) =1.
\]
Henceforth, $\pi^*_2\beta_n = \epsilon$ and the proof is complete.
\end{proof}

The following lemma plays the role of \cite[Proposition~3.5(ii)]{Bea09-Brauer_of_Eriques_surf}: we remark how the proof of Beauville makes use of the fact that \(\tilde q\) is non-degenerate, which is false as soon as \(n>1\).
\begin{lem}\label[lem]{pi_*}
There exists an element $\lambda_0\in H^2(S^{[n]},\Z)$ such that $\pi_*\lambda_0 = K_{E_n}$. 
\end{lem}
\begin{proof}
If we apply \cite[Theorem 5.5]{Aguilar_Prieto} to the two squares in diagram~\eqref{ch6_comm_squares-new}, we get the following commutative squares relating push-forward and pull-back maps
\begin{equation}\label{ch6_cohom_squares-new}
\begin{tikzcd}
{H^2(S^{[n]},\Z)} \arrow[d, "\pi_*"] & {H^2(S^{(n)},\Z)} \arrow[d, "\pi^{(n)}_*"] \arrow[l, "h^*"', hook'] \arrow[r, "\Delta'^*"] & {H^2(S,\Z)} \arrow[d, "\pi_{E,*}", two heads] \\
{H^2(E_n,\Z)}                          & {H^2\left(E_{(n)},\Z\right)} \arrow[l, "f^*"] \arrow[r, "g^*"', hook]                   & {H^2(E,\Z)}.                      
\end{tikzcd}
\end{equation}
The morphism $\pi_{E, *}$ is surjective by \cite[Proposition 3.5(ii)]{Bea09-Brauer_of_Eriques_surf}. 
\cref{lem:g*} shows that $g^*$ is injective and has image $nH^2(E, \Z)$ and \cref{lem_basta_svp} shows that $\Delta'^*$ has image $nH^2(S, \Z)$. Therefore, the morphism $\pi_*^{(n)}$ is surjective. 

Let $\alpha \in H^2(S^{(n)},\Z)$ be a class such that $\pi_*^{(n)}(\alpha) = K_{E_{(n)}}$. By diagram chasing in diagram~\eqref{ch6_cohom_squares-new}  we obtain 
\begin{equation}\label{eqn:preimageK}
    \pi_*(h^*\alpha)=f^* K_{E_{(n)}}.
\end{equation}

The morphism $f$ is birational, hence we can write \(K_{E_n}=f^*K+D\) for some divisor \(D\) supported on the exceptional locus \(\Exc(f)\). We want to show $D=0$. Since \(\pi\) is flat we have the equality of Weil divisors 
\[\pi^*[\Exc(f)]=[\pi^{-1}\Exc(h)]=\delta \in \Pic(S^{[n]}),\]
which implies that \(D\) is not torsion in \(\Pic(E_n)\). Therefore, 
it is enough to show \(\pi^*D=0\). The commutativity of left-side square of diagram~\eqref{ch6_comm_squares-new}, together with the equalities $\pi^{(n)^*}K_{E_{(n)}} = \cO_{S^{(n)}}$ and $\pi^*K_{E_n} = \cO_{S^{[n]}}$, gives the equality
\[\pi^* f^*K_{E_{(n)}}=h^*\pi^{(n)^*}K_{E_{(n)}}=h^*\cO_{S^{(n)}}=\cO_{S^{[n]}}=\pi^* f^*K_{E_{(n)}} +\pi^*D,\]
from which we obtain $\pi^*D = 0$. This proves $f^* K_{E_{(n)}} = K_{E_n}$. The proof of the lemma follows from  ~\eqref{eqn:preimageK} by taking \(\lambda_0 \coloneqq h^*\alpha\).
\end{proof}

\begin{prop}\label[prop]{kerPushF}
    Consider the push-forward $\pi_{*,2}:H^2\left(S^{[n]},\Z/2\right)\longrightarrow H^2\left(T,\Z/2\right)$. Restricted to the two-dimensional vector space $U_2$, its kernel is $\mathrm{Ker}(\pi_{*,2}) = \{0,\epsilon\}$.        
\end{prop}
\begin{proof}
    With the notation of the proof of \cref{kerImPullB}, we have $\pi^*_2\beta_n = \epsilon$, hence $\pi_{*,2}\epsilon =2\beta_n=0$. By \cref{pi_*}, there exists an element $\lambda_0\in H^2(S^{[n]},\Z)$ such that $\pi_*\lambda_0 = K_T$.
    In particular, 
    \[
        0=\pi^* K_T =\pi^*\pi_*\lambda_0=(1+\sigma^*)\lambda_0.
    \]
  The proof of \cref{thm_brauer_computations} shows that $H^1\left(\Z/2,H^2(S^{[n]},\Z)\right)\simeq U_2$, and $\lambda_0$ defines an element in this quotient whose push-forward is non-zero. Therefore, $\mathrm{Ker}\left((\pi_{*,2})_{|U_2}\right) = \{0,\epsilon\}$.
\end{proof}
\begin{cor}\label[cor]{corEquivalenceNumerical}
    Let $\lambda\in H^2(S^{[n]},\Z)$, the following are equivalent:
    \begin{enumerate}[label = \rm(\arabic*)]
        \item\label{item1-sec6} $\pi_*\lambda=0$ and $\lambda\not\in (1-\sigma^*)\left(H^2(S^{[n]},\Z)\right)$;
        \item\label{item2-sec6} $\sigma^*\lambda=-\lambda$ and $q(\lambda)\equiv 2 \pmod{4}$.
    \end{enumerate}
\end{cor}
\begin{proof}
    We write $\lambda=(a,a',m,b)\in F \oplus F\oplus \Z\left\langle\delta\right\rangle \oplus U$,  and denote $\overline{b}\in U_2$ the class of $b$. 
    
    We first show \ref{item1-sec6}$\Rightarrow$\ref{item2-sec6}. First observe that $\lambda$ lies in the kernel of $(1+\sigma^*)=\pi^*\circ\pi_*$, hence it must be of the form $(a, -a, 0, b)$ for some $a\in F$ and $b\in U$. Note that if $\overline{b}=0$, then $b=2\Tilde{b}$ for some $\tilde{b}\in U$ and $\lambda=(1-\sigma^*)(a,0,0,\Tilde{b})$, which is a contradiction. Hence, $\overline{b}$ is a non-zero element of $\mathrm{Ker}(\pi_{*,2})$ and is therefore equal to $\epsilon$ by \cref{kerPushF}. In particular, 
    \[
    q(\lambda) = 2q(a)+q(b)\equiv q(b) \equiv 2 \pmod{4}
    \]
    (recall that the quadratic form in $U_2$ is the one of $U$ divided by $2$).
    
    We now prove \ref{item2-sec6}$\Rightarrow$\ref{item1-sec6}. Again, $\lambda$ is of the form $(a,-a,0,b)$, for some $a\in F$ and some $b\in U$ such that
    \[
    q(b)\equiv q(\lambda)\equiv 2\pmod{4}.
    \]
    In particular, the vector $\overline b\in U_2$ has square $1$, which implies $\overline{b}=\epsilon$.  As before, we have that $\lambda\not\in (1-\sigma^*)\left(H^2(S^{[n]},\Z)\right)$. Moreover, $\pi_*\lambda=\pi_*\sigma^*\lambda=-\pi_*\lambda$, so that $\pi_*\lambda=0$ or $\pi_*\lambda=K_T$. Since $\pi_{*,2}\overline{b} = 0$ by \cref{kerPushF}, we obtain $\pi_*\lambda=0$.
\end{proof}

\begin{cor}\label[cor]{corEquivalenceBrauerClass}
    Let $b_{n}\in \mathrm{Br}(E_n)$ be the non-zero class. The following are equivalent:
    \begin{enumerate}
        \item $\pi^*b_n=0$ in $\mathrm{Br}(S^{[n]})$;
        \item it exists $L\in \mathrm{Pic}(S^{[n]})$ such that $\pi_*c_1(L)=0$ and $c_1(L)\not\in (1-\sigma^*)\left(H^2(S^{[n]},\Z)\right)$;
        \item it exists $L\in \mathrm{Pic}(S^{[n]})$ such that $\sigma^*L=L^\vee$ and $q(L)\equiv 2 \pmod{4}$.
    \end{enumerate}
\end{cor}
\begin{proof}
    It follows directly from the general result \cite[Corollary 4.3]{Bea09-Brauer_of_Eriques_surf} applied to the quotient $\pi:S^{[n]}\to E_n$ and by \cref{corEquivalenceNumerical}.
\end{proof}

It is possible to translate the conditions in \cref{corEquivalenceBrauerClass} into a moduli statement. For this purpose, we recall the theory of the period map for Enriques manifolds, developed by Oguiso and Schr\"oer in \cite{Oguiso-Schroer_Periods}. With the notation introduced in  ~\eqref{eqn:latticeSn}, let $$L \coloneqq  F \oplus F \oplus (-2(n-1)) \oplus U$$ be the abstract Beauville--Bogomolov lattice for \hk\ manifolds of type $K3^{[n]}$, endowed with the $\Z/2\Z$-action defined in~\eqref{eqn:involution}. 
Let $L^-$ be the anti-invariant sublattice of $L$, 
and consider the \emph{period domain}
\[
\Omega \coloneqq  \{ \omega \in \mathbb{P}(L_\mathbb{C}^-) \mid \omega^2 = 0 \text{ and } \omega\cdot \overline{\omega}>0\}.
\]
Given an Enriques manifold $E_n = S^{[n]}/\langle i^{[n]}\rangle$ and an isometry $\varphi\colon H^2(S^{[n]}, \Z) \to L$ respecting the $\Z/2\Z$-actions, the point $\left[\varphi_\C\left(H^{2,0}(S^{[n]})\right)\right]$ lies in $\Omega$.

Finally, for any $\lambda\in L^-$, we denote by $H_{\lambda}$ the hyperplane in $\Omega$ cut out by $\P(\lambda^\perp\otimes \C)$.

\begin{prop}
For an Enriques manifold $E_n$ we have $\pi^*b_n =0$ if and only if its period $[\varphi(\omega_{S^{[n]}})]$ belongs to $H_\lambda$ for $\lambda\in L^-$ such that $\lambda^2 \equiv 2 \pmod{4}$.
\end{prop}
\begin{proof}
The elements in $L^-$ orthogonal to $[\varphi(\omega_{S^{[n]}})]$ are precisely the ones coming from $\mathrm{Pic}(S^{[n]})$. By \cref{corEquivalenceBrauerClass}, $\pi^*b_n =0$ is equivalent to say that $[\varphi(\omega_{S^{[n]}})]$ belongs to an hyperplane $H_\lambda$ with $\lambda \in \varphi(c_1(\mathrm{Pic}(S^{[n]})))$ such that $\sigma^*(\lambda) = -\lambda$ and $\lambda^2 \equiv 2 \pmod{4}$.
\end{proof}

Let $\Gamma$ be the group of isometries of $L^-$. Following \cite[(6.4)]{Bea09-Brauer_of_Eriques_surf}, the quotient $\cM = \Omega/\Gamma$ is a coarse moduli space for Enriques manifolds of type $E_n$.
Note that the lattice $L^-$  is orthogonal to the component $(-2(n-1))$ of $L$, corresponding to the exceptional class $\delta$. In particular, $L^-$ is isometric to the anti-invariant sublattice of a K3 surface $S$ carrying a fixed point free involution. Therefore, $\cM$ coincides with coarse moduli space of Enriques surfaces. The same argument as in \cite[Corollary 6.5] {Bea09-Brauer_of_Eriques_surf} gives the analogous statement for Enriques manifolds of type $E_n$.
\begin{thm}\label[thm]{corEquivalenceBrauerClassModuli}
The Enriques manifolds $E_n$ such that $\pi^*b_n=0$ form an infinite countable union of non-empty hypersurfaces in the moduli space $\cM$.
\end{thm}

\section{Special Brauer-Severi varieties over Enriques manifolds}\label[section]{secSBvar}

Let \(T\) be the Enriques manifold of index \(d\)  and \(\pi \colon X\to T\) its \hk \ universal cover; more generally, the constructions in this section will carry over to the case where \(X\to T\) is a finite étale cyclic cover of smooth projective manifolds such that \(h^{2,0}(T)=0\) and \(h^{1,0}(X)=0\). 
The aim of this section is to construct (non-trivial) Brauer--Severi varieties $P\to T$ in the kernel of the pullback map
$$
\pi^*_{\Br}\colon \Br(T)\to \Br(X).
$$
Let $\sigma$ be a generator of the Galois group $G$ of $\pi$. By \cite[Proposition 4.1]{Bea09-Brauer_of_Eriques_surf}, the kernel of $\pi^*_{\Br}$ is isomorphic to $$\ker(\pi_*|_{\NS(X)})/(\id-\sigma^*)(\NS(X)),$$
where $\pi_*|_{\NS(X)}$ is the restriction of the push-forward map $\pi_*\colon H^2(X, \Z)\to H^2(T, \Z)$ to $\NS(X)$.

Denote by $\Pic_G(X)$ the group of \(G\)-linearized line bundles modulo \(G\)-equivariant isomorphism. 
There exists a morphism
$$
N\colon \Pic(X) \lra \Pic_G(X),
$$
sending a line bundle $L$ on $X$ to the line bundle
$$
N_d(L)\coloneqq  L \otimes \sigma^*L \otimes \cdots \otimes (\sigma^{d-1})^*L
$$
endowed with then  natural \(G\)-linearization given by cyclic permutation of elementary tensors and the fact that \((\sigma^d)^* L =L\). 

By \cite[Lemma 3.3.1]{Brion_notes} the pullback via $\pi$ induces an isomorphism
\[\pi_G^* \colon \Pic(T) \to \Pic_G(X)
\] which sends \(F\) to \(\pi^*F\) endowed with the \(G\)-linearization induced by the pull-back diagram
\[
\begin{tikzcd}
\pi^*F \arrow[rrd, bend left] \arrow[rdd, bend right] \arrow[rd, "\simHalf"{pos=0.35, yshift=-6pt}] &  & \\
& \sigma^*\pi^*F \arrow[d] \arrow[r] & F \arrow[d] \\
& X \arrow[r, "\pi\circ\sigma = \pi"']     & T.        
\end{tikzcd}
\]

Since taking the first Chern class induces an isomorphism $\Pic(T)\simeq H^2(T,\Z)$, we can apply this to $F = \pi_*c_1(L)$ for $L\in\mathrm{Pic}(X)$. In this case, the induced linearization on 
\[
\pi^*F = L \otimes \sigma^*L \otimes \cdots \otimes (\sigma^{d-1})^*L
\]
corresponds to the one induced by $N$. In other words, the composition
$$
\Pic(X)\xrightarrow{\pi_*\circ c_1} \Pic(T) \xrightarrow{\pi_G^*} \Pic_G(X)
$$
is equal to the norm map $N$. Since $\pi_G^*$ is an isomorphism and $c_1$ is injective, we obtain that the kernel of $N$ is isomorphic to the the kernel of $\pi_*\vert_{\rm{NS}(X)}$.

\begin{defin}
For any line bundle \(L\in \Pic(X)\), for $k\geq 1$ we define its \textit{k-th partial norm} to be the line bundle 
\[N_k(L)\coloneqq  \bigotimes_{i=0}^{k-1} (\sigma^i)^*L \in \Pic(X).\]
Moreover, we set $N_0(L)\coloneqq  \mathcal{O}_X$.
\end{defin} 
\begin{prop}\label[proposition]{GLinBastaaaaaa}
    Let $L$ be line bundle contained in the kernel of $N$, and consider the the rank \(d\) vector bundle 
    $$
    \cE \coloneqq  \bigoplus_{k=0}^{d-1} N_k(L). 
    $$
    The projective bundle $p\colon \P_X(\cE)\to X$ descends to $T$. Namely, there exist a Brauer-Severi variety $P_L\to T$ over $T$ and a cartesian diagram
    \[\begin{tikzcd}
	\P_X(\cE) & {P_L} \\
	X & T
	\arrow[ from=1-1, to=1-2]
	\arrow["p"', from=1-1, to=2-1]
	\arrow[ from=1-2, to=2-2]
	\arrow["\pi"', from=2-1, to=2-2]
\end{tikzcd}\]
\end{prop}
Before proving the proposition, we discuss some properties of lines bundles in the kernel of $N$. By definition of $\Pic_G(X)$, a line bundle $L\in \Pic(X)$ lies in $\ker(N)$ if and only if  there exists a $G$-equivariant isomorphism 
\begin{equation}\label{eqn:morphH}
    h\colon \cO_X\lra N_d(L)
\end{equation}
such that the following diagram is commutative
    \begin{equation}\label{eqn:TrivialLin}
    \begin{tikzcd}
    \cO_X\ar[r,"h"]\ar[d, "\sigma^*"'] & N_d(L)\ar[d, "\eta"]\\
    \sigma^*\cO_X\ar[r, "\sigma^*h"'] & \sigma^*N_d(L),
\end{tikzcd}
\end{equation}
where $\eta$ is the natural $G$-linearization on $N_d(L)$.
The following lemma will be useful for the proof.
\begin{lem}\label[lemma]{lemmaSec}
Let $L$ be a line bundle in $\ker(N)$ and consider the identification $L \otimes \sigma^*N_d=N_d\otimes L$ given by the equality $(\sigma^d)^* = \id$.
    The morphism
\begin{equation}\label{eqn:varphi}
        \psi\colon L\xrightarrow{\id_L\otimes h} L\otimes N_d \xrightarrow{\id_L\otimes\eta}L \otimes \sigma^*N_d=N_d\otimes L    
\end{equation}
is equal to the morphism $h\otimes \id_L$ 
\end{lem}
\begin{proof}
Since the statement is local on $X$, we can restrict to an open subset $U$ of $X$ on which the line bundles $L, \sigma^*L, \cdots, (\sigma^{d-1})^*L$ are trivial. Given trivializing section $s_i$ of $(\sigma^i)^*L\vert_U$,  
    the section $t=s_0\otimes s_1\otimes\cdots\otimes s_{d-1}$ is a trivializing section of $N_d(L)$. Therefore, we can write $h(1) = ut$ for some $u\in\cO_X(U)^*$, and
    $$
        \eta(h(1)) = u\,\eta(t) = u\,s_1\otimes\cdots \otimes s_{d-1}\otimes s_0,
    $$
    where we used that $\eta$ is given by the cyclic permutation of the factors of $N_d(L)$.

    Now compute the morphism $\psi$ on the local section \(s_0\in L(U)\):
    $$
        s_0\longmapsto s_0\otimes \eta(h(1)) = u\,s_0\otimes s_1\otimes\cdots \otimes s_{d-1}\otimes s_0 = u\,t \otimes s_0,
    $$  
    which is equal to $h(1)\otimes s = (h\otimes\id_L)(s)$.
\end{proof}

\begin{proof}[Proof of \cref{GLinBastaaaaaa}]
        Since the morphism $X\to T$ is induced by the action of the group $G =\langle \sigma\rangle$, 
    by  descent theory (\cite[Exposé VIII, Corollaire 7.6]{SGA1}), a morphism $p\colon Q\to X$ descends to $T$ if and only if $p\colon Q\to X$ admits a $G$-linearization, which is a morphism ${\Phi\colon Q\to \sigma^*Q}$ such that the diagram
    \[\begin{tikzcd}
	{Q} & {\sigma^*Q} \\
	    & X
	\arrow["\Phi", from=1-1, to=1-2]
	\arrow["\sigma^*p", from=1-2, to=2-2]
	\arrow["p"', from=1-1, to=2-2]
\end{tikzcd}\]
is commutative, and the composition
$$
(\sigma^{d-1})^*\Phi\circ \cdots \sigma^*\Phi\circ \Phi\colon Q\lra (\sigma^d)^*Q =Q
$$
is the identity.

We construct a $G$-linearization $\Phi\colon \P_X(\cE) \to \sigma^*\P_X(\cE)$. Note that we have isomorphisms of projective bundles over $X$
$$
\sigma^*\P_X(\cE)\simeq \P_X(\sigma^*\cE) \simeq \P_X(L\otimes \sigma^*\cE),
$$
where the second isomorphism is the canonical identification induced by the line bundle $L$ (defined locally by multiplying for a trivializing section of $L$).

Note that for all \(k\in \{1,\dots d-1\}\), we have the equality
\begin{equation}\label{eq_partial_norms_translation}
    L\otimes \sigma^* N_k(L)=N_{k+1}(L).
\end{equation}
Hence, the vector bundle $L\otimes \sigma^*\cE$ is equal to
\begin{equation}\label{eqn:Lsigma^*E}
    L\otimes \sigma^*\cE = (L\otimes \sigma^*\cO_X) \oplus N_2(L)\oplus \cdots \oplus N_{d-1}(L)\oplus N_d(L).
\end{equation}

Consider the morphism of vector bundles $\phi_1\colon \cE \to L\otimes \sigma^*\cE$ defined as
\[\begin{tikzcd}[column sep = small]
	\cO_X & \oplus & {L} & \oplus & {N_2} & \oplus & \cdots & \oplus & {N_{d-2}} & \oplus & {N_{d-1}} \\
	\\
    \\
	{L\otimes \sigma^*\cO_X} & \oplus & {N_2} & \oplus & {N_3} & \oplus & \cdots & \oplus & {N_{d-1}} & \oplus & {N_{d}},
	\arrow["h",color=red, from=1-1, to=4-11, in=140, out=320]
	\arrow["\id_L\otimes \sigma^*"'{pos=0.6},color=myblue, from=1-3, to=4-1]
	\arrow["\id_{N_2}"'{pos=0.3}, from=1-5, to=4-3]
	\arrow["\id_{N_{d-1}}"'{pos=0.3}, from=1-11, to=4-9]
\end{tikzcd}\]
    where $h$ is the $G$-isomorphism defined in \cref{eqn:morphH}, and we wrote $N_i$ for the line bundle $N_i(L)$.
    The morphism $\phi_1$ is an isomorphism of vector bundles over $X$, therefore it induces a morphism of projective bundles $\P(\phi_1)\colon \P_X(\cE)\to\P(L\otimes \sigma^*\cE)$ over $X$.
    Finally, we set $\Phi$ the morphism
    $$
        \Phi\colon \P_X(\cE)\xrightarrow{\P(\phi_1)}\P(L\otimes \sigma^*\cE)= \P(\sigma^*\cE).
    $$
    We show that $\Phi$ is a $G$-linearization, namely that the composition 
    $$
        (\sigma^{d-1})^*\Phi\circ \cdots \sigma^*\Phi \circ \Phi
    $$
    is the identity. Denote by $\Phi_k$ the partial composition 
    $$(\sigma^{k-1})^*\Phi\circ \cdots \sigma^*\Phi \circ \Phi\colon \P_X(\cE)\lra \P((\sigma^{k})^*\cE).
    $$
    Using \cref{eq_partial_norms_translation} recursively, and the fact that $\sigma^d=\id$ we can write
    \begin{multline}\label{eqn:Nksigma^k^*E}
    N_k(L)\otimes (\sigma^k)^*\cE = (N_k(L)\otimes (\sigma^k)^*\cO_X) \oplus N_{k+1}(L)\oplus \cdots \oplus N_d(L)\\ \oplus (N_d(L)\otimes N_1(L))
    \oplus \cdots \oplus(N_d(L)\otimes N_{k-1}(L))
\end{multline}

    We claim that 
    the morphism $\Phi_k$ factors as
    \begin{equation}\label{ind_hyp}
        \Phi_k\colon \P_X(\cE) \xrightarrow{\P(\phi_k)} \P(N_k(L)\otimes (\sigma^k)^*\cE)\simeq \P((\sigma^k)^*\cE),
    \end{equation}
    where $\phi_k$ is the isomorphism of vector bundles defined as
\begin{equation}\label{eqn:defphi_k}
\begin{tikzcd}[column sep =0mm]
	\cO_X & \oplus & {L} & \oplus & \cdots & \oplus & {N_{k-1}} & \oplus & N_k & \oplus & \cdots & \oplus & {N_{d-1}} \\
	\\
	\\
    {N_k\otimes (\sigma^k)^{*}\cO_X} & \oplus & \cdots & \oplus & {N_{d-1}} & {\oplus } & {N_d} & \oplus & {N_d\otimes L} & \oplus & \cdots & \oplus & {N_{d}\otimes N_{k-1}}.
	\arrow["h"{pos=0.1}, color=red, from=1-1, to=4-7, out = 310, in = 130]
	\arrow["h\otimes \id_L"{pos=0.2}, color=purple, from=1-3, to=4-9,  out = 310, in = 130]
    \arrow["h\otimes \id_{N_{k-1}}"{pos=0.2}, color=purple, from=1-7, to=4-13,  out = 310, in = 130]
	\arrow["\id_{N_k}\otimes (\sigma^k)^*"'{pos=0.8}, color=myblue, from=1-9, to=4-1]
	\arrow["\id_{N_{d-1}}"{pos=0.2}, from=1-13, to=4-5]
\end{tikzcd}
\end{equation}
Once the claim is proved, taking $k=d$ yields the wanted result. Indeed, this proves that the composition $\Phi_d$ is equal to the composition
$$
\P_X(\cE)\xrightarrow{\P(\phi_d)}\P(N_d(L)\otimes \cE)\simeq\P_X(\cE),
$$
where the morphism $\phi_d$ is the isomorphism of vector bundles $h\otimes \id_\cE$. This shows in particular that $\Phi_d =\id_{\P_X(\cE)}$ and $\Phi$ is a $G$-linearization of $\P_X(\cE)$. By the above discussion, this proves that $\P_X(\cE)\to X$ is the pullback of a Brauer-Severi variety $P_L\to T$.\\

We now prove the claim by induction on $k$. 
The case $k=1$ follows by definition of $\Phi$. Consider the morphism $\Phi_{k+1}$, and suppose that the factorisation~\eqref{ind_hyp} is true for the morphism $\Phi_k$. Note that there is equality
$$
\Phi_{k+1} = \sigma^*\left((\sigma^{k-1})^*\Phi\circ \cdots \circ \Phi\right)\circ \Phi = \sigma^*\Phi_k \circ \Phi,
$$
inducing a factorisation
\begin{equation}\label{bigDiagPhi_k}
    \begin{tikzcd}
    \P_X(\cE) \ar[rr, "\Phi"]\ar[rrd, "\P(\phi_1)"'] && \P(\sigma^*\cE)  \ar[r, "\sigma^*\Phi_k"]\ar[rd, "\P(\sigma^*\phi_k)"'] & \P\left(\sigma^*((\sigma^k)^*\cE)\right)\\
    && \P(L\otimes \sigma^*\cE)\ar[u, "\vsim"']\ar[rd, "\P(\id_L\otimes \sigma^*\phi_k)"'] & \P\left(\sigma^*(N_k(L)\otimes (\sigma^k)^*\cE)\right)\ar[u, "\vsim"']\\
    && & \P\left(L\otimes \sigma^*(N_k(L)\otimes (\sigma^k)^*\cE)\right)\ar[u, "\vsim"'],
\end{tikzcd}
\end{equation}
where we used factorisation~\eqref{ind_hyp} for the morphism $\Phi_k$, and the vertical isomorphisms are induced by locally multiplying for a section of the line bundle $L$.
We prove that the composition 
\begin{equation*}
    \cE\xrightarrow{\phi_1}L\otimes \sigma^*\cE\xrightarrow{\id_L\otimes \sigma^*\phi_k}L\otimes \sigma^*(N_k(L)\otimes (\sigma^k)^*\cE)
\end{equation*}
is equal to the morphism $\phi_{k+1}$.

The vector bundle $L\otimes \sigma^*(N_k(L)\otimes (\sigma^k)^*\cE)$ decomposes as
\begin{align*}
    L\otimes \sigma^*(N_k\otimes (\sigma^k)^*\cE)&= (L \otimes \sigma^*N_k\otimes (\sigma^{k+1})^*\cO_X)\oplus\bigoplus_{i=k+1}^{d-1} L\otimes \sigma^*N_i \oplus \bigoplus_{i=0}^{k-1} L\otimes \sigma^*N_d\otimes \sigma^*N_i\\
    &=N_{k+1}\otimes (\sigma^{k+1})^*\cO_X \oplus \bigoplus_{i=k+2}^d N_i \oplus \bigoplus_{i=1}^k  N_d\otimes N_i
\end{align*}

By definition of $\phi_k$ (see~\eqref{eqn:defphi_k}), for $0<i<k$  the restriction of $\id_L\otimes \sigma^*\phi_k$ to the term $N_{i+1}(L) = L\otimes\sigma^*(N_i(L))$ of $L\otimes \sigma^*\cE$ is equal to 
$$
\id_L\otimes \sigma^*(h\otimes\id_{N_i(L)})\colon L \otimes \sigma^*(\cO_X\otimes N_i(L)) \lra L \otimes \sigma^*N_d\otimes \sigma^*(N_i(L))
$$ 
Using the identification $\id_L\otimes \sigma^*\colon L\to L\otimes \sigma^*\cO_X$, and commutativity of diagram~\eqref{eqn:TrivialLin}, the above map is given by
$$
L \otimes \sigma^*N_i\xrightarrow{\psi\otimes \id_{N_i}} L \otimes \sigma^*N_d \otimes N_i = N_d\otimes N_{i+1},
$$
where $\psi$ was defined in~\eqref{eqn:varphi}. By \cref{lemmaSec}, this is exactly the morphism $h\otimes \id_{N_{i+1}}$.

Therefore, we can write down the composition $\id_L\otimes \sigma^*\phi_k\circ \phi_1$ as
$$
\begin{tikzcd}[column sep =0mm]
    \cO_X &\oplus & L &\oplus &  N_2 & \oplus &  \cdots  &\oplus & N_{k} & \oplus & N_{k+1} & \oplus & \cdots  \,\oplus\, {N_{d-1}}\\
    \\
    \\
	L\otimes \sigma^*\cO_X & \oplus & {N_2} & \oplus & \cdots & \oplus & {N_{k}} & \oplus & N_{k+1} & \oplus & \cdots & \oplus & {N_{d}} \\
	\\
	\\
    {N_{k+1}\otimes (\sigma^{k+1})^{*}\cO_X} & \oplus & \cdots & \oplus & {N_{d}} & {\oplus } & {L\otimes \sigma^*N_d} & \oplus & {N_d\otimes N_2} & \oplus & \cdots & \oplus & {N_{d}\otimes N_k}.
	\arrow["h",color=red, from=1-1, to=4-13, in=140, out=320]
	\arrow["\id_L\otimes \sigma^*"'{pos=0.6},color=myblue, from=1-3, to=4-1]
	\arrow["\id_{N_2}"'{pos=0.3}, from=1-5, to=4-3]
    \arrow["\id_{N_{k}}"'{pos=0.3}, from=1-9, to=4-7]
	\arrow["\id_{N_{k+1}}"'{pos=0.3}, from=1-11, to=4-9]
	\arrow["\id_L\otimes\sigma^*h"{pos=0.1}, color=red, from=4-1, to=7-7, out = 310, in = 130]
	\arrow["h\otimes \id_{N_2}"{pos=0.2}, color=purple, from=4-3, to=7-9,  out = 310, in = 130]
    \arrow["h\otimes \id_{N_{k}}"{pos=0.2}, color=purple, from=4-7, to=7-13,  out = 310, in = 130]
	\arrow["\id_{N_{k+1}}\otimes (\sigma^{k+1})^*"'{pos=0.8}, color=myblue, from=4-9, to=7-1]
	\arrow["\id_{N_{d}}"{pos=0.2}, from=4-13, to=7-5]
\end{tikzcd}
$$

Using again \cref{lemmaSec} and commutativity of diagram~\eqref{eqn:TrivialLin}, we obtain that the morphism 
$$
L \xrightarrow{\id_L\otimes \sigma^*} L\otimes \sigma^*\cO_X \xrightarrow{\id_L\otimes \sigma^*h} L\otimes \sigma^*N_d = N_d\otimes L
$$
is the morphism $h\otimes \id_L$. This shows that the composition  $\id_L\otimes \sigma^*\phi_k\circ \phi_1$ is exactly the morphism $\phi_{k+1}$ given by~\eqref{ind_hyp}, and concludes the proof.
\end{proof}

The following result completes the geometrical description of $\ker(\pi^*_{\Br})$.  

\begin{prop}\label[prop]{charactSBV}
    Let \(L\in \Pic(X)\) contained in \(\ker(N)\). In the above notation, we have that
    \[P_L \simeq \Proj(\cF) \textrm{ for some   rank d locally free sheaf } \cF \iff c_1(L)\in (1-\sigma^*)(H^2(X,\Z))\]
\end{prop}
\begin{proof}
Suppose \(c_1(L)\in (1-\sigma^*)(H^2(X,\Z))\). Since \(1-\sigma^*\) is a morphism of Hodge structures and \(H^1(X,\cO_X)=0\), this is equivalent to the fact there exists \(M\in \Pic(X)\) such that 
\[L\simeq M\otimes \sigma^* M^\vee .\]

Since \(N_k(M\otimes \sigma^*M^\vee)\simeq M\otimes  (\sigma^*)^{k} M^\vee\),  we deduce that 
\[ M^\vee \otimes \cE_L\simeq   M^\vee \oplus \sigma^* M^\vee \oplus  (\sigma^*)^2M^\vee \oplus \cdots \oplus (\sigma^*)^{d-1}M^\vee \simeq \pi^* \pi_* M^\vee ,\]

where the last isomorphism can be easily checked in the analytic category, by GAGA. Ultimately, we have \(\sigma\)-equivariant isomorphisms
$$
\Proj_X(\cE_L)\simeq \Proj_X(\pi^*\pi_* M)\simeq \Proj_T(\pi_* M)\times_T X.
$$

Since the quotient by \(\langle \phi\rangle\) of left and right-hand terms respectively are \(P_L\) and \(\Proj_T(\pi_*M)\), by uniqueness of categorical quotient we have the claim with \(\cF\coloneqq  \pi_*M\).

Conversely, suppose there exists some locally free sheaf \(\cF\) of rank \(d\) on \(T\) such that \(P_L\simeq \Proj_T(\cF)\). Then
\(\Proj_X(\cE_L)\simeq \Proj_X(\pi^*\cF)\), hence there exists a line bundle \(M\) on \(X\) such that
\[\pi^*\cF \simeq \cE_L\otimes M.\]

Clearly \(\pi^*\cF\simeq \sigma^*\pi^* \cF\), which implies the first isomorphism in
\[\cE_L \otimes M\simeq \sigma^* \cE_L \otimes \sigma^*M\simeq (\sigma^* \cE_L \otimes L)\otimes (L^\vee \otimes \sigma^*M).\]
By \eqref{eq_partial_norms_translation}, \(\cE_L\simeq \sigma^*\cE_L \otimes L\) and hence the above yields
\[\cE_L\otimes M \simeq \cE_L\otimes (L^\vee \otimes \sigma^*M).\]
Taking the determinant on both sides we get \(M^{\otimes d}\simeq (L^\vee \otimes \sigma^* M)^{\otimes d}\), and conclude that
\[c_1(L)= (1-\sigma^*)(c_1(M^\vee))\]
because \(H^2(X,\Z)\) is torsion free.
\end{proof}

\bibliographystyle{amsalpha-abbrv}
\bibliography{biblio}

@article{Dold,
 author = {Dold, Albrecht},
 title = {Decomposition theorems for {{\(S(n)\)}}-complexes},
 fjournal = {Annals of Mathematics. Second Series},
 journal = {Ann. Math. (2)},
 issn = {0003-486X},
 volume = {75},
 pages = {8--16},
 year = {1962},
 doi = {10.2307/1970415},
 zbMATH = {3202602},
 Zbl = {0125.01201}
}

@article{BNWS10-DefEnriques,
 author = {Boissi{\`e}re, Samuel and Nieper-Wi{\ss}kirchen, Marc and Sarti, Alessandra},
 title = {Higher dimensional {Enriques} varieties and automorphisms of generalized {Kummer} varieties},
 fjournal = {Journal de Math{\'e}matiques Pures et Appliqu{\'e}es. Neuvi{\`e}me S{\'e}rie},
 journal = {J. Math. Pures Appl. (9)},
 issn = {0021-7824},
 volume = {95},
 number = {5},
 pages = {553--563},
 year = {2011},
 doi = {10.1016/j.matpur.2010.12.003},
 zbMATH = {5888512},
 Zbl = {1215.14046}
}

@article{Bea09-Brauer_of_Eriques_surf,
 author = {Beauville, Arnaud},
 title = {On the {Brauer} group of {Enriques} surfaces},
 fjournal = {Mathematical Research Letters},
 journal = {Math. Res. Lett.},
 issn = {1073-2780},
 volume = {16},
 number = {5-6},
 pages = {927--934},
 year = {2009},
 doi = {10.4310/MRL.2009.v16.n6.a1},
 zbMATH = {5707692},
 Zbl = {1195.14053}
}

@article{OguisoSchroer2011,
 author = {Oguiso, Keiji and Schr{\"o}er, Stefan},
 title = {Enriques manifolds},
 fjournal = {Journal f{\"u}r die Reine und Angewandte Mathematik},
 journal = {J. Reine Angew. Math.},
 issn = {0075-4102},
 volume = {661},
 pages = {215--235},
 year = {2011},
 doi = {10.1515/CRELLE.2011.077},
 zbMATH = {5995790},
 Zbl = {1272.14026}
}

@article{Oguiso-Schroer_Periods,
 author = {Oguiso, Keiji and Schr{\"o}er, Stefan},
 title = {Periods of {Enriques} manifolds},
 fjournal = {Pure and Applied Mathematics Quarterly},
 journal = {Pure Appl. Math. Q.},
 issn = {1558-8599},
 volume = {7},
 number = {4},
 pages = {1631--1656},
 year = {2011},
 doi = {10.4310/PAMQ.2011.v7.n4.a25},
 zbMATH = {6107792},
 Zbl = {1316.32011}
}

@book{Weibel-intro_to_hom_algebra,
 author = {Weibel, Charles A.},
 title = {An introduction to homological algebra},
 fseries = {Cambridge Studies in Advanced Mathematics},
 series = {Camb. Stud. Adv. Math.},
 volume = {38},
 isbn = {0-521-43500-5},
 year = {1994},
 publisher = {Cambridge: Cambridge University Press},
 zbMATH = {595200},
 Zbl = {0797.18001}
}

@article{KapferMenet2018,
 author = {Kapfer, Simon and Menet, Gr{\'e}goire},
 title = {Integral cohomology of the generalized {Kummer} fourfold},
 fjournal = {Algebraic Geometry},
 journal = {Algebr. Geom.},
 issn = {2313-1691},
 volume = {5},
 number = {5},
 pages = {523--567},
 year = {2018},
 doi = {10.14231/AG-2018-014},
 zbMATH = {6999243},
 Zbl = {1423.14243}
}

@book{Voisin2002I,
 author = {Voisin, Claire},
 title = {Hodge theory and complex algebraic geometry. {I}. {Translated} from the {French} by {Leila} {Schneps}},
 fseries = {Cambridge Studies in Advanced Mathematics},
 series = {Camb. Stud. Adv. Math.},
 volume = {76},
 isbn = {0-521-80260-1},
 year = {2002},
 publisher = {Cambridge: Cambridge University Press},
 zbMATH = {1822310},
 Zbl = {1005.14002}
}

@article{Beauville83,
 author = {Beauville, Arnaud},
 title = {Vari{\'e}t{\'e}s k{\"a}hleriennes dont la premi{\`e}re classe de {Chern} est nulle},
 fjournal = {Journal of Differential Geometry},
 journal = {J. Differ. Geom.},
 issn = {0022-040X},
 volume = {18},
 pages = {755--782},
 year = {1983},
 doi = {10.4310/jdg/1214438181},
 zbMATH = {3853948},
 Zbl = {0537.53056}
}

@book{Hatcher,
 author = {Hatcher, Allen},
 title = {Algebraic topology},
 isbn = {0-521-79540-0},
 year = {2002},
 publisher = {Cambridge: Cambridge University Press},
 zbMATH = {2103273},
 Zbl = {1044.55001}
}

@misc{BGGG25,
  title = {Non-Existence of {{Enriques}} Manifolds from {{OG10}} Type Manifolds},
  author = {Billi, Simone and Giovenzana, Franco and Giovenzana, Luca and Grossi, Annalisa},
  year = 2025,
  number = {arXiv:2501.06893},
  eprint = {2501.06893},
  publisher = {arXiv},
  doi = {10.48550/ARXIV.2501.06893},
  url = {https://arxiv.org/abs/2501.06893},
  urldate = {2025-12-30},
  archiveprefix = {arXiv},
  copyright = {Creative Commons Attribution 4.0 International},
  langid = {english}
}

@misc{PacSar23,
 author = {Pacienza, Gianluca and Sarti, Alessandra},
 title = {On the cone conjecture for {Enriques} manifolds},
 year = {2023},
 howpublished = {Preprint, {arXiv}:2303.07095 [math.{AG}] (2023)},
 url = {https://arxiv.org/abs/2303.07095},
 arXiv = {arXiv:2303.07095}
}

@article{Aguilar_Prieto,
 author = {Aguilar, Marcelo A. and Prieto, Carlos},
 title = {Transfers for ramified covering maps in homology and cohomology},
 fjournal = {International Journal of Mathematics and Mathematical Sciences},
 journal = {Int. J. Math. Math. Sci.},
 issn = {0161-1712},
 volume = {2006},
 number = {11},
 pages = {94651, 28},
 year = {2006},
 doi = {10.1155/IJMMS/2006/94651},
 url = {https://eudml.org/doc/53291},
 zbMATH = {5164297},
 Zbl = {1117.55012}
}

@article{Schroer_Brauer_group,
 author = {Schr{\"o}er, Stefan},
 title = {Topological methods for complex-analytic {Brauer} groups},
 fjournal = {Topology},
 journal = {Topology},
 issn = {0040-9383},
 volume = {44},
 number = {5},
 pages = {875--894},
 year = {2005},
 doi = {10.1016/j.top.2005.02.005},
 zbMATH = {2198606},
 Zbl = {1083.14019}
}

@article{DRTX26,
  title = {{{MMP}} for {{Enriques}} Pairs and Singular {{Enriques}} Varieties},
  author = {Denisi, Francesco Antonio and R{\'i}os Ortiz, {\'A}ngel David and Tsakanikas, Nikolaos and Xie, Zhixin},
  year = 2026,
  journal = {J. Ec. polytech. Math.},
  volume = {13},
  pages = {629--686},
  issn = {2429-7100, 2270-518X},
  doi = {10.5802/jep.334},
  url = {https://jep.centre-mersenne.org/articles/10.5802/jep.334/},
  urldate = {2026-04-13},
  langid = {english}
}

@incollection{Brion_notes,
 author = {Brion, Michel},
 title = {Linearization of algebraic group actions},
 booktitle = {Handbook of group actions. Volume IV},
 isbn = {978-1-57146-365-4},
 pages = {291--340},
 year = {2018},
 publisher = {Somerville, MA: International Press; Beijing: Higher Education Press},
 zbMATH = {7034435},
 Zbl = {1410.14038}
}

@article{Martinez,
 author = {Mart{\'{\i}}nez, Hermes},
 title = {The {Brauer} group of {{\(K3\)}} covers},
 fjournal = {Revista Colombiana de Matem{\'a}ticas},
 journal = {Rev. Colomb. Mat.},
 issn = {0034-7426},
 volume = {46},
 number = {2},
 pages = {185--204},
 year = {2012},
 url = {www.scielo.org.co/pdf/rcm/v46n2/v46n2a05.pdf},
 zbMATH = {6643186},
 Zbl = {1350.14029}
}

@book{SGA1,
 author = {Grothendieck, Alexander},
 title = {S{\'e}minaire de g{\'e}om{\'e}trie alg{\'e}brique du {Bois} {Marie} 1960/61 ({SGA} 1), dirig{\'e} par {Alexander} {Grothendieck}. {Augment{\'e}} de deux expos{\'e}s de {M}. {Raynaud}. {Rev{\^e}tements} {\'e}tales et groupe fondamental. {Expos{\'e}s} {I} {\`a} {XIII}. ({Seminar} on algebraic geometry at {Bois} {Marie} 1960/61 ({SGA} 1), directed by {Alexander} {Grothendieck}. {Enlarged} by two reports of {M}. {Raynaud}. {{\`E}tale} coverings and fundamental group)},
 fseries = {Lecture Notes in Mathematics},
 series = {Lect. Notes Math.},
 issn = {0075-8434},
 volume = {224},
 year = {1971},
 publisher = {Springer, Cham},
 language = {French},
 doi = {10.1007/BFb0058656},
 zbMATH = {3370468},
 Zbl = {0234.14002}
}

@misc{Hartlieb_Verni25,
  title = {On the Topological {{Brauer}} Group of Generalized {{Kummer}} Varieties},
  author = {Hartlieb, Moritz and Verni, Matteo},
  year = 2025,
  number = {arXiv:2512.14262},
  eprint = {2512.14262},
  primaryclass = {math},
  publisher = {arXiv},
  doi = {10.48550/arXiv.2512.14262},
  url = {http://arxiv.org/abs/2512.14262},
  urldate = {2026-04-13},
  archiveprefix = {arXiv}
}

@article{Markman2007_integral,
 author = {Markman, Eyal},
 title = {Integral generators for the cohomology ring of moduli spaces of sheaves over {Poisson} surfaces},
 fjournal = {Advances in Mathematics},
 journal = {Adv. Math.},
 issn = {0001-8708},
 volume = {208},
 number = {2},
 pages = {622--646},
 year = {2007},
 doi = {10.1016/j.aim.2006.03.006},
 zbMATH = {5083644},
 Zbl = {1115.14036}
}

@article{Totaro_2020_integral,
 author = {Totaro, Burt},
 title = {The integral cohomology of the {Hilbert} scheme of points on a surface},
 fjournal = {Forum of Mathematics, Sigma},
 journal = {Forum Math. Sigma},
 issn = {2050-5094},
 volume = {8},
 pages = {6},
 year = {2020},
 doi = {10.1017/fms.2020.35},
 zbMATH = {7276276},
 Zbl = {1451.14011}
}

@misc{Nakaoka57,
 author = {Nakaoka, Minoru},
 title = {Cohomology of symmetric products},
 year = {1957},
 language = {English},
 howpublished = {J. {Inst}. {Polytechn}., {Osaka} {City} {Univ}., {Ser}. {A} 8, 121-145 (1957).},
 zbMATH = {3132352},
 Zbl = {0080.38103}
}

@article{Gugnin2012,
 author = {Gugnin, D. V.},
 title = {Topological applications of graded {Frobenius} {{\(n\)}}-homomorphisms. {II}},
 fjournal = {Transactions of the Moscow Mathematical Society},
 journal = {Trans. Mosc. Math. Soc.},
 issn = {0077-1554},
 volume = {2012},
 pages = {167--182},
 year = {2012},
 language = {English},
 doi = {10.1090/S0077-1554-2013-00201-0},
 zbMATH = {6187012},
 Zbl = {1273.13002}
}

@book{Voisin2002II,
 author = {Voisin, Claire},
 title = {Hodge theory and complex algebraic geometry. {II}. {Transl}. from the {French} by {Leila} {Schneps}},
 fseries = {Cambridge Studies in Advanced Mathematics},
 series = {Camb. Stud. Adv. Math.},
 volume = {77},
 isbn = {0-521-80283-0},
 year = {2003},
 publisher = {Cambridge: Cambridge University Press},
 language = {English},
 zbMATH = {1966201},
 Zbl = {1032.14002}
}
\end{document}